\documentclass[11pt]{article}
\usepackage{booktabs}
\usepackage{hyperref}
\usepackage{url}
\usepackage{color, colortbl}
\usepackage{graphicx}
\usepackage{amsmath,amssymb,amsthm, bm}
\usepackage{algorithm}
\usepackage{algpseudocode}
\usepackage{xcolor}
\usepackage{flushend}
\usepackage{subfigure}
\usepackage[none]{hyphenat}
\usepackage{pifont} 
\newcommand{\cmark}{\ding{51}} 
\newcommand{\xmark}{\ding{55}} 

\def\hL{\hat{L}}
\def\b1{\mathbf 1}

\def\bbx{\bar{\mathbf x}}
\def\bby{\bar{\mathbf y}}
\def\tbbx{\bar{\widetilde{\mathbf x}}}

\def\bbv{\bar{\mathbf v}}

\def\bz{\mathbf z}
\def\bv{\mathbf v}
\def\tbx{\widetilde{\mathbf x}}
\def\tby{\widetilde{\mathbf y}}

\def\bx{\mathbf x}

\def\by{\mathbf y}

\def\bu{\mathbf u}
\def\bw{\mathbf w}
\def\bA{\mathbf A}

\def\bL{\mathbf L}

\def\bI{\mathbf I}

\def\bW{\mathbf W}

\def\bphi{\boldsymbol \phi}

\def\cL{\mathcal{L}}

\newtheorem{lemma}{Lemma}

\newtheorem{claim}{Claim}
\newtheorem{theorem}{Theorem}

\newtheorem{assumption}{Assumption}
\newtheorem{Corollary}{Corollary}

\usepackage{amsmath}
\usepackage{accents}

\def\remark{\addtocounter{remark}{1}\def\@currentlabel{\theremark}%
	\emph{Remark~\theremark}. } \makeatother
\newcounter{remark}

\title{\LARGE On the Divergence of Decentralized Non-Convex Optimization}
\author{Mingyi Hong$^\dag$, Siliang Zeng$^\star$, Junyu Zhang$\dag\dag$, and Haoran Sun$^\dag$ \thanks{
	$^\dag$ Department of ECE, University of Minnesota Twin Cities, Minneapolis, MN USA
$^\star$ School of Science and Engineering, Chinese University of Hong Kong, Shenzhen.
$^{\dag\dag}$ Department of Industrial and Systems Engineering, University of Minnesota Twin Cities, Minneapolis, MN, USA. MH, HS are supported in part by NSF under Grant CMMI-172775, CIF-1910385 and by AFOSR under grant 19RT0424, ARO under grant W911NF-19-1-0247.}}

\begin{document}

\maketitle

\begin{abstract}%
  In this work, we study a generic class of decentralized algorithms in which $N$ agents jointly optimize the non-convex objective function $f(\bu):=1/N\sum_{i=1}^{N}f_i(\bu)$, while only communicating with their neighbors. This class of problems has become popular in modeling many signal processing and decentralized machine learning applications, and efficient algorithms have been proposed for such a type of problem. However, by constructing a series of counter-examples, we show that when certain {\it local} Lipschitz conditions (LLC) on the local function gradient $\nabla f_i$'s are not satisfied, most of the existing decentralized algorithms diverge, even if the global Lipschitz condition (GLC) is satisfied, where the sum function $f$ has Lipschitz gradient.  This observation brings out a fundamental theoretical issue of the existing decentralized algorithms -- their convergence conditions are {\it strictly stronger} than centralized algorithms such as the gradient descent (GD), which only requires the GLC. 
  Additionally, this observation raises an important open question: How to design decentralized algorithms when the LLC, or even the GLC, is not satisfied? 
  
  To address the above question, we design a first-order algorithm called {\it \underline{M}ulti-st\underline{a}ge \underline{g}radi\underline{en}t \underline{t}r{a}cking \underline{a}lgorithm} (MAGENTA), 
  which is capable of computing stationary solutions of the original problem with neither the LLC nor the GLC condition. {In particular, we show that the proposed algorithm converges sublinearly to certain $\epsilon$-stationary solution, where the precise rate depends on various algorithmic and problem parameters. In particular, if the local function $f_i$'s are $Q$th order polynomials, then the rate becomes $\mathcal{O}(1/\epsilon^{Q-1})$. Such a rate is {\it tight} for the special case of $Q=2$ where each $f_i$ satisfies LLC.} To our knowledge, this is the first attempt that studies decentralized non-convex optimization problems with neither the LLC nor the GLC.
\end{abstract}

\section{Introduction}
Decentralized optimization has received significant attention from the research community. It has become the power horse not only in the traditional signal processing applications, but also in modern applications such as decentralized machine learning, and training neural networks \cite{lian2017can, nedic2009distributed, chen2012diffusion, yuan2016convergence}. 

Despite the fact that much research has been done in this area, there is still a significant lack of understanding about the behavior and performance of various kinds of decentralized algorithms, especially when the problem under consideration is {\it non-convex}.  Compared with {\it centralized} algorithms in which all the problem data is located at the same place,  decentralized algorithm has the limitation that each computation node only has a {\it local view} about the entire problem. Therefore, compared with centralized schemes, it is much more difficult to properly design and analyze decentralized algorithms. In particular, compared to centralized algorithms, decentralized algorithms require much stronger assumptions to guarantee the convergence.   For example, consider the following standard  formulation for decentralized optimization:
\begin{align}\label{eq:problem:central}
\min \; \frac{1}{N}\sum_{i=1}^{N} f_i(\bu): = f(\bu)
\end{align}
where each $f_i: \mathbb{R}^K\to \mathbb{R}$ is a continuously differentiable function that can only be accessed by the local agent $i$. 
Assume there are $N$ agents in the system, where each agent $i$ has access to the {\it local function} $f_i(\cdot)$. The above problem can then be rewritten equivalently in the following form:
\begin{align}\label{eq:problem:main}
\min\; \frac{1}{N}\sum_{i=1}^{N} f_i(\bx_i) : = g(\bx), \quad \mbox{s.t.}~{\bx_i = \bx_j}, \; \mbox{\rm if~$(i,j)$~ are neighbors},
\end{align}
where $\bx:=[\bx_1; \cdots, \bx_k]\in\mathbb{R}^{NK}$. 
One typical assumption made for (almost) all decentralized algorithms is that, each component function $f_i(\cdot)$ has the Lipschitz gradient. That is, the following Local Lipschitz Condition (LLC) is satisfied:
\begin{align}\label{eq:Lip:distributed}
\|\nabla f_i(\bv_i)- \nabla f_i(\bw_i)\|\le L_i \|\bv_i-\bw_i\|, \quad \forall~\bv_i, \bw_i \in \mbox{dom}(f_i), \; \forall~i. \quad \mbox{\rm\bf(LLC)}
\end{align}
On the other hand, if we view problem \eqref{eq:problem:central} as a centralized problem, then the vanilla gradient descent (GD) algorithm can compute first-order stationary solutions, by just assuming the following Global Lipschitz Condition (GLC)  for the {\it sum} function $f(\bx)$:
\begin{align}\label{eq:Lip:central}
\|\nabla f(\by)- \nabla f(\bz)\|\le L\|\by-\bz\|, \quad \forall~\bz, \by \in \mbox{dom}(f). \quad \mbox{\rm\bf(GLC)}
\end{align}
It can be shown that LLC is stronger, and in fact in many situations {\it strictly} so, than the GLC \eqref{eq:Lip:central}; see Sec. \ref{sec:global_local} for detailed discussions. For example, let $N=3$,  and the local functions are non-convex, given  $f_1(x_1)= y^3,  f_2(x_1) = -y^3, f_3(x_3) = y^2$, and let $f(u)=\frac{1}{3}(f_1(u)+f_2(y)+f_3(u))$, then it is easy to see that GLC is satisfied but the LLC is not. 

The above discussion leads to a number of open research questions. First, it appears that a problem that can be handled perfectly by a {\it centralized} algorithm {\it may} turn out to be extremely challenging for decentralized algorithms, because the {\it strong} assumption of LLC can fail to hold. It is unclear if such a theoretical gap between centralized and decentralized algorithms is fundamental, or it can be addressed by better algorithm design and/or analysis. 
Second, in many learning problems,  the local and/or global Lipschitz constants may not even exist over the entire domain $\mathbb{R}^{K}$.  For example, it is well-known that the cost functions of matrix factorization type problems, or multi-layer neural networks, do not have global Lipschitz gradient \cite{li2019provable,zhang2019gradient-clipping}. In particular, it is shown in  \cite[Sec. H]{zhang2019gradient-clipping}, that for training AWD-LSTM neural networks,  the gradient Lipschitz constants vary greatly across different training iterations , and they are not bounded. 
However, there has been no existing decentralized optimization algorithms that can effectively deal with these practical situations. 

\noindent{\bf Contributions.} The main contributions of this work are given below:

\noindent{\bf (1)} We revisit convergence guarantees for a number of popular decentralized first-order algorithms (which utilize local first-order oracle for computation, and use certain double-stochastic weight matrix to perform local communication) for distributed non-convex optimization, and identify that the LLC is not only sufficient, but also {\it necessary} for their global convergence; 

\noindent{\bf (2)} We design a new distributed algorithm 
that requires neither LLC nor GLC, while being able to achieve global sublinear convergence to some $\epsilon$-stationary solutions;

\noindent{\bf (3)} We analyze the convergence rates (to certain properly defined $\epsilon$-stationarity solution), and identify its dependency on certain ``growth rate" of the local gradients as well as a number of algorithmic parameters; Additionally, we show that when $f_i$'s belong to a class of non-convex $Q$th order polynomial functions (which include various types of matrix factorization problems as special cases), the convergence rate  is $\mathcal{O}(1/\epsilon^{Q-1})$, and such a rate reduces to the known result when $Q=2$, in which case the LLC and GLC both hold true.

Overall, we hope that our study will reveal some  insights about the state-of-the-art distributed algorithms, especially when they are applied to non-convex problems. Further, we expect that the proposed algorithms and the associated analysis will be of independent interest to the research community, because they can be potentially adopted by different types of first-order methods (e.g., decentralized, deterministic, or stochastic methods) to better deal with the lack of Lipschitz gradients. To the best of our knowledge, this work provides the first algorithm, and the sharpest rate available, to deal with the decentralized problem \eqref{eq:problem:main} with neither the GLC nor the LLC assumption.

\noindent{\bf Notations.} We use $\mathbb{B}(v, \bx)$ to denote a ball centered at $\bx$ and has radius $v$; For a given compact set $X$ and a given point, we use $\mbox{dist}(\bx, X)$ to denote the minimum distance between $\bx$ and any point $\bz$ in $X$, that is: 
$$\mbox{dist}(\bx, X): = \min_{\bz\in X}\|\bx-\bz\|,$$
where $\|\cdot\|$ denotes the vector $\ell_2$ norm. The notation $\b1$ is used to denote an all one vector, and  for a vector $\bx:=[\bx_1;\cdots; \bx_N]\in\mathbb{R}^{NK}$, the notation $\bar{\bx}\in\mathbb{R}^K$ is used to denote the average of the components, that is $\bar{\bx}:=\frac{1}{N}\sum_{i=1}^{N}\bx_i$.  For a given matrix $X$ we use $\lambda{X}$ to denote its maximum eigenvalue.

\section{State-of-the-art Algorithms and Convergence Conditions}\label{sec:review}
In this section, we set the stage by reviewing a number of representative state-of-the-art algorithms for distributed non-convex optimization, and discuss their convergence conditions. Note that due to space limitation, we can only discuss a small subset of existing distributed algorithms. The readers are referred to a recent issue of IEEE Signal Processing Magazine for more in-depth discussion of many other related algorithms \cite{distributedSPM}. For  simplicity of notation, we will assume that $K=1$, that is, the optimization variable $\bu$  in \eqref{eq:problem:central} is a scalar. 

Our focus will be given to the setting where the agents form an undirected graph $G = ({\cal V}, {\cal E})$, with  ${\cal V}$ being the set of agents, and ${\cal E}$ representing the communication pattern among them. 
The \emph{graph incidence matrix} $\bA \in \mathbb{R}^{|{\cal{E}}| \times N}$ has $A_{ei}=1$, $A_{ej}=-1$ if $j>i$, $e=(i,j)\in \cal{E}$, and $A_{ek} = 0$ for all $k \in {\cal V} \setminus \{i,j\}$. Note that $\bA^\top \bA: =\bL_G \in \mathbb{R}^{N \times N}$ is the \emph{graph Laplacian} matrix. 
Lastly, a symmetric \emph{mixing matrix} ${\bm W}$ satisfies the following conditions:
\begin{align}\label{eq:W}
\mathrm{Range}\{\bI_N -\bW\}=\text{{\rm span}} \{\mathbf{1}\};~  -\bI_N \preceq \bW \preceq  \bI_N;~W_{ij}=0~\text{if}~(i,j)\notin {\cal{E}}, W_{ij}>0 \; \;\text{otherwise}. 
\end{align}

Let us briefly describe a few typical assumptions needed to ensure convergence of decentralized algorithms. Besides the GLC and LLC in \eqref{eq:Lip:central} and \eqref{eq:Lip:distributed}, the following lower boundedness assumption is often needed. 
Specifically, it is typical to assume that either the local function $f_i(\bx_i)$'s are lower bounded, or the average function $f(\bu)$ is lower bounded by a finite number $\underline{f}$: 
	\begin{align} \label{eq:lower:bound}
	f_i(\bx_i) > \underline{f}, \; \forall~\bx_i \in \mathbb{R}^K,\;  \mbox{\rm\bf(LLB)},\quad 
	f(\bu) > \underline{f}, \; \forall ~\bu \in \mathbb{R}^K, \; \mbox{\rm\bf(GLB)}.
	\end{align}

Next, we review a few state-of-the-art distributed non-convex algorithms to solve problem \eqref{eq:problem:main}. Our focus will be given to {\it deterministic} algorithms in which full local gradients are utilized when performing local computations. 

 First, the Distributed Gradient Descent (DGD) algorithm, expressed below, is a simple and popular algorithm for both convex and non-convex problems (where $g(\cdot)$ is defined in \eqref{eq:problem:main}) 	\begin{align}\label{eq:dgd}
	\bx^{r+1} =  \bW \bx^r - \alpha^r \nabla g(\bx^r).
	\end{align}
	It is originally proposed by \cite{nedic2009distributed} for convex problems, and recently analyzed by \cite{zeng2018nonconvex} for non-convex problems. To guarantee convergence of DGD to a stationary solution, it needs the (LLB) condition, and that the stepsize satisfies (\cite[Theorem 2]{Zeng19distributedGD}):
	\begin{align}\label{eq:stepsize}
	\sum_{r=1}^{\infty}\alpha^r = \infty, \; \sum_{r=1}^{\infty}(\alpha^r)^2 < \infty, \; \alpha^r \neq 0.
	\end{align}

A closely related family of algorithms are called the Gradient Tracking (GT) algorithms \cite{nedic2017achieving,di2016next, gnsd19}, which uses constant stepsizes. Specifically, the updates of GT is given below, where $\by^r$ is a sequence that tracks $\nabla g(\bx)$: 
	\begin{align*}
	\bx^{r+1} &=\bW  \bx^r -\alpha \by^r,\quad \by^{r+1} =\bW \by^r + \nabla g(\bx^{r+1}) - \nabla g(\bx^{r}).
	\end{align*}
	Note that the above updates can also be simplified as following:
	\begin{align}\label{eq:gt}
	\bx^{r+1} &=2\bW  \bx^{r} - \bW^2 \bx^{r-1} - \alpha \nabla g(\bx^{r}) + \alpha \nabla g(\bx^{r-1}),
	\end{align}
	which shows that GT shares some similarity to the EXTRA algorithm \cite{shi2015extra} developed for convex problems; see discussion in \cite[Section 2.2.1]{nedic2017achieving}. 
	
	 Another family of algorithm is designed by using the Primal-Dual (PD) strategy. 
	 According to the definitions of the \emph{graph incidence matrix} $\bA$, we can observe that the  consensus constraint  in problem \eqref{eq:problem:main} can be reformulated as a set of linear equalities, i.e.,
	$$\bx_i = \bx_j, \; \mbox{\rm if~$(i,j)$~ are neighbors},  \Leftrightarrow \bA\bx = 0.$$
	Therefore, we can derive the following Augmented Lagrangian from the problem \eqref{eq:problem:main}:
	\begin{align}\label{proxPDA_AL_function}
	\cL(\bx, \lambda) = g(\bx) + \langle \lambda, \bA\bx \rangle + \frac{\rho}{2} \|\bA\bx\|^2
	\end{align}
	where $\lambda \in  \mathbb{R}^{E}$ is the dual variable of the constraint $\bA\bx = 0$, and $\rho >0$ is a penalty parameter. Then the proximal primal-dual algorithm (Prox-PDA) proposed in \cite{hong2017prox}  performs a primal step which minimizes a linearized version of \eqref{proxPDA_AL_function}, and then a dual gradient ascent step:
	\begin{subequations}\label{eq:pd:iteration}
		\begin{align}
		\bx^{r+1}& =\arg\min_{\bx}\;  \langle   \nabla_{\bx}\cL(\bx^r,\lambda^r), \bx-\bx^r\rangle   + \left(\frac{\beta}{2} +  \frac{\rho \lambda_{\max}(\bA^T \bA)}{2} \right) \|\bx-\bx^r\|^2 \label{eq:pd:iteration1} \\
		\lambda^{r+1}& = \lambda^r +\rho \bA \bx^{r+1} \label{eq:pd:iteration2}.
		\end{align}
	\end{subequations}
	By subtracting the two consecutive updates of $\bx$, one can cancel the $\lambda^r$ sequence, and simplify the Prox-PDA algorithm as 
	\begin{align}\label{eq:pd}
	\bx^{r+1}    = \left(\bI_N -\frac{\rho \bA^T \bA   }{  \beta +   \rho \lambda_{\max}(\bA^T \bA)  } \right) \left(2\bx^r	  -\bx^{r-1} \right) -\frac{1}{\beta +   \rho \lambda_{\max}(\bA^T \bA)}(\nabla g(\bx^r)-\nabla g(\bx^{r-1})).
	\end{align}

Other related algorithms include different types of diffusion based schemes, which 
follow the adapt-then-combine (ATC) scheme \cite{lopes2008diffusion,cattivelli2009diffusion}. Instead of first performing the consensus then the gradient descent step, the ATC based methods first perform the gradient descent steps, and such a strategy leads to a number of variations of DGD and GT \cite{bianchi2013performance,sayed2014adaptation,pu2018distributed,sun2019improving}. For more detailed survey of existing algorithms, we refer the readers to \cite{chang2020distributed}. In Table \ref{fig:table_compare}, we summarize a number of state-of-the-art algorithms for distributed non-convex optimization, their convergence conditions and convergence rates \footnote{Note that we have also included a number of stochastic algorithms, such as $D^2$, and the algorithm proposed in \cite{bianchi2013convergence}. However, the main focus of this paper will be on the deterministic algorithms.}. Note that the convergence rate here refers to the number of iterations required to achieve certain $\epsilon$-stationary solutions (to be  defined shortly in subsequent sections). For the proposed algorithms, their convergence rates are dependent on certain growth function $\mbox{gr}(\cdot)$ which characterizes how fast the local/global Lipschitz constants grow with the size of the feasible set, and this function does not have an explicit expression for generic functions. See Sec. \ref{sub:convergence:1} for discussion. 

 \begin{table}[t]
	\vspace{-2px}
	\caption{Comparison of algorithms on  decentralized non-convex optimization}
	\label{fig:table_compare}
	\footnotesize
	\begin{center}
		\begin{sc}
			\begin{tabular}{lccccccccc}
				\toprule
				Algorithm  & LLC & GLC & LLB & GLB  &stepsize & convergence rate\\
				\midrule
				 $[$Tsitsiklis-86$]$ \cite{tsitsiklis86} & \cmark & \cmark & \cmark & \cmark & dim. & N/A\\
				$[$Bianchi-13$]$\cite{bianchi2013convergence} &\cmark  & \cmark & \cmark & \cmark&dim.  & N/A \\
				NEXT \cite{di2016next} & \cmark  & \cmark & \xmark & \cmark &dim.   & N/A \\
				DGD \cite{zeng2018nonconvex} & \cmark  & \cmark & \cmark & \cmark & dim.   & $\mathcal{O}(\epsilon^{-2})$\\
				DeFW \cite{wai2017decentralized} & \cmark  & \cmark & \cmark & \cmark & dim.   & $\mathcal{O}(\epsilon^{-2})$\\
				D$^2$ \cite{tang2018d} & \cmark  & \cmark & \cmark & \cmark & dim.   & $\mathcal{O}(\epsilon^{-2})$\\
				Prox-PDA \cite{hong2017prox}& \cmark  & \cmark & \cmark &\cmark & const. & $\mathcal{O}(\epsilon^{-1})$ \\
				xFILTER \cite{sun2018distributed} & \cmark  & \cmark & \cmark &\cmark & const.&   $\mathcal{O}(\epsilon^{-1})$ \\
				SONATA \cite{sun2019convergence}& \cmark  & \cmark & \xmark & \cmark & const.  & $\mathcal{O}(\epsilon^{-1})$ \\
				 {\bf Proposed}  & \xmark  & \xmark & \xmark &\cmark&const. & $\mathcal{O}((\epsilon^{-1})\mbox{gr}(\mbox{poly}(\epsilon^{-1})))$ \\
				\bottomrule
			\end{tabular}
		\end{sc}
	\end{center}
\end{table}

\section{On the  Local Lipschitz Conditions (LLC)}\label{sec:global_local}
We note that the algorithms surveyed in the previous section all require both the GLC \eqref{eq:Lip:central} and the LLC  \eqref{eq:Lip:distributed}. 
In fact, this statement holds true for all the algorithms reviewed in a recent survey \cite{chang2020distributed}. On the other hand, for the centralized smooth optimization problems, it is well-known that even the {\it simplest} gradient descent (GD) algorithm is capable of computing first-order stationary solutions only under the GLC \cite{Nesterov04}. A fundamental question then arises: 

\vspace{0.1cm}
\noindent\fcolorbox{black}[rgb]{0.95,0.95,0.95}{\begin{minipage}{0.98\columnwidth}
		\begin{center}
		{\bf (Q1)}	Is LLC \eqref{eq:Lip:distributed} necessary for decentralized algorithms to work for problem \eqref{eq:problem:main}?
		\end{center}
\end{minipage}}
\vspace{0.1cm}

In this section, we provide analysis on the LLC, discuss specific scenarios when it will or will not hold, and show why its presence is crucial for the convergence of many existing decentralized algorithms. 

\subsection{Divergence of State-of-the-art Decentralized Algorithms Without LLC}\label{sub:divergence}

First, we analyze relationships between the GLC and LLC.  Detailed proof is in Appendix \ref{proof:equivalence}. 
\begin{lemma}\label{lemma:equivalence}
	Suppose that each component function $f_i$ is convex. Then the GLC \eqref{eq:Lip:central} is equivalent to the LLC \eqref{eq:Lip:distributed}. That is, if the LLC \eqref{eq:Lip:distributed} holds for a set of $L_i's$, then GLC \eqref{eq:Lip:central} holds with
	$L:=\frac{1}{N}\sum_{i=1}^{N}L_i$; Conversely, if the GLC \eqref{eq:Lip:central} holds true, then there must exist a set of finite positive numbers $\{L_i\}$'s such that the LLC \eqref{eq:Lip:distributed} holds true. 
\end{lemma}

The next result says that when we allow each $f_i$'s to be non-convex, then the GLC no longer implies the LLC. This is easy to see, for example, we can consider the following simple problem: $f(u) = u^2 + u^q - u^q$, for any $q\ge 3$, and let $f_1(u) = u^2$, $f_2(u) = u^q$ and $f_3(u) = -u^q$.

\begin{lemma}\label{lemma:non-equivalence}
	Suppose that each $f_i$ is not necessarily convex. Then there exist problems in the form of \eqref{eq:problem:central}, where GLC \eqref{eq:Lip:central} holds but the LLC \eqref{eq:Lip:distributed} does not hold.
\end{lemma}

From the above result, it is clear that the LLC is more restrictive than the GLC. Meanwhile, our discussion in Sec. \ref{sec:review} suggests that the LLC is {\it sufficient} to guarantee convergence of all the algorithms surveyed therein. It is then natural to ask, whether we can relax the LLC to the GLC for these algorithms? 
Next, we give a negative answer, by showing that 																																																																																																																																																																				three state-of-the-art algorithms diverge (to infinity), if only the GLC but not the LLC is satisfied. 
The detailed proofs can be found in Appendix \ref{proof:dgd} -- \ref{proof:pd}.

\begin{claim}\label{claim:dgd}
	Consider the DGD iteration \eqref{eq:dgd}, and fix $\alpha^r =\alpha>0$ for all $r$. There exists a problem instance satisfying the GLC \eqref{eq:Lip:central}, and a mixing matrix $\bW$ satisfying \eqref{eq:W}, such that no matter how small the constant stepsize is,  one can find an initial solution $\bx^0\in\mathbb{R}^{K}$ that makes DGD diverge (to infinity). {Additionally, if we replace the constant stepsize to the following diminishing stepsize $\alpha^r = \alpha/(1+r)$ (which satisfies \eqref{eq:stepsize}), then for any $\alpha>0$, there exists an initial solution $\bx^0\in\mathbb{R}^{K}$,  such that the DGD still diverges.}
\end{claim}

\vspace{-0.3cm}
\begin{claim}\label{claim:gt}
	Consider the GT iteration \eqref{eq:gt}. There exists a problem instance satisfying the GLC \eqref{eq:Lip:central}, and an associated mixing matrix $\bW$ satisfying \eqref{eq:W}, such that no matter how small the stepsize $\alpha$ is,  one can find an initial solution $\bx^0\in\mathbb{R}^K$ that makes the GT iteration diverge.
\end{claim}
\vspace{-0.3cm}

\begin{claim}\label{claim:pd}
	Consider the Prox-PDA iteration \eqref{eq:pd}. There exists a problem instance satisfying the GLC \eqref{eq:Lip:central},  such that for any $\rho>0$,  one can find an initial solution $\bx^0\in\mathbb{R}^K$ that makes the Prox-PDA iteration diverge.
\end{claim}
\vspace{-0.3cm}

\vspace{-0.3cm}
\section{The Proposed Algorithm}\label{sec:proposed}
Through the discussions in the previous section, we see concrete examples in which the lack of the LLC makes many state-of-the-art decentralized algorithms diverge. In fact, by using similar constructions in Claim \ref{claim:dgd} -- \ref{claim:pd}, it is easy to show that the lack of GLC will result in similar divergence behaviors. Despite the fact that we did not illustrate such divergence behaviors for {\it all existing} decentralized algorithms, we hope that the readers can see the need to deal with such undesirable situations (from both theoretical and practical standpoint). 

One may wonder at this point, that is it even possible to relax the LLC and/or GLC condition for decentralized algorithms? More formally, we have the following research question:

\vspace{0.1cm}
\noindent\fcolorbox{black}[rgb]{0.95,0.95,0.95}{\begin{minipage}{0.98\columnwidth}
		\begin{center}
			{\bf (Q2)}	With neither LLC nor GLC, how to design convergent decentralized algorithms for problem \eqref{eq:problem:main}?
		\end{center}
\end{minipage}}
\vspace{0.1cm}

In this section, we develop and analyze new algorithms capable of computing first-order stationary solutions with {\it neither} LLC nor GLC. 

\subsection{Preliminaries}

To begin our discussion, let us first state the assumptions we need to analyze problem \eqref{eq:problem:central}. 

\begin{assumption}\label{ass:problem}
{\bf (About the problem.)}	We assume the following: 
	\begin{enumerate}
		\vspace{-0.1cm}
		\item For any given compact set $X\subset  \mathbb{R}^K$, the average function $f(\bu)$ has Lipschitz gradient, that is:
		\begin{align}\label{eq:f:lip}
		\|\nabla f(\by) -\nabla f(\bz)\|\le L(X)\|\by-\bz\|, \; \forall~\by,\bz\in X;
		\end{align}
		Further, for any given compact set $X\subset \mathbb{R}^K$, each component function has Lipschitz gradient:
		\begin{align}\label{eq:fi:lip}
		\|\nabla f_i(\by) -\nabla f_i(\bz)\|\le L_i(X)\|\by-\bz\|, \; \forall~\by, \bz\in X,
		\end{align}
	where $L(X)\in (0, \infty)$ and $L_i(X)\in (0, \infty)$ are finite constants dependent on the set $X$. Additionally, if $X_1\subseteq X_2\subset \mathbb{R}^K$, then $L(X_1)\le L(X_2)$, $L_i(X_1)\le L_i(X_2), \forall~i$; Without loss of generality, we assume that $L_i(X)\ge 1, \forall~i, \; \forall~X$.  Define 
	$\hat{L}(X):=\max_{i}L_i(X).$ Clearly we have $\hat{L}\ge L$. 
		\vspace{-0.1cm}
		\item The average function $f(\bu)$ is lower bounded by $\underline{f}$: That is, the GLB condition in \eqref{eq:lower:bound} holds.
		\vspace{-0.1cm}
		\item The domain of each of the local functions satisfies $\mbox{dom}(f_i)\equiv \mathbb{R}, \; \forall~i$; 
	\end{enumerate}
\end{assumption} 

	\vspace{-0.2cm}

It is easy to see that the assumptions made above are much more relaxed compared with the GLC and LLC. First,  the conditions \eqref{eq:f:lip} and \eqref{eq:fi:lip} suggest that the average function $f(\by)$ as well as the local functions are no longer required to have Lipschitz gradients over the entire space $\mathbb{R}^K$. Instead, what is essentially needed is that the functions are continuously differentiable. Second, the local functions can have wider choices than the global function, for example $f_i(x)=x^3$ is allowed, but $f(u)=u^3$ is not allowed since it is not lower bounded on $\mathbb{R}$. Third, although the present paper only focuses on {\it unconstrained} problems in which the domain is the entire space, the algorithms and analysis developed here can be easily extended to to {\it constrained} problems where each local node has a {\it shared} constraint set $\bx_i\in X, \; \forall~i$, for some closed convex set $X\subset \mathbb{R}^K$. Finally, it is important to note that  neither the average function $f(\by)$ nor the local functions $f_i(x)$'s is required to have {\it compact} level sets. That is, the following sets do not need to be compact for any constant $c$:
\begin{align}\label{eq:level:set}
Y:= \{\by: f(\by)\le c\}, \; X_i:=\{\bx_i: f_i(\bx_i)\le c\}.
\end{align}
Therefore the considered problem class covers important problems such as matrix factorization problems.  
See the examples provided below. 

\noindent{\bf Examples.} First, note that the counter examples used to show Claims \ref{claim:dgd} -- \ref{claim:pd} all satisfy Assumption \ref{ass:problem}. Second, consider the following distributed matrix factorization problem \cite{Daneshmand18}
	\begin{align}
	f_i(\bu,\bw_i) = \|\bu^T \bw_i - v_i\|^2, \; \forall~i,
	\end{align}
	where $\bu$ and $\bw_i$ is the optimization variable (i.e., the shared dictionary and the local coefficients, respectively), and $v_i$ is the observed data. Clearly the above problem satisfies Assumption \ref{ass:problem}. 
In fact, one can generalize the problem to the matrix setting (in which the dictionary $\bu$ is a matrix), or change the objective to some generic distances such as the $\ell_q$ distance given below (for any $q>1$)
	\begin{align}
f_i(\bu,\bw_i) = \|\bu^T \bw_i - v_i\|_q^q, \; \forall~i.
\end{align}

	Moreover, consider the following polynomial function with even order
	\begin{align}\label{eq:poly:even}
	f_i(\bu) = \sum_{w} \sigma_{i,w} u^{w_1}_1\times \cdots \times u^{w_K}_K,
	\end{align}
	where the sum is over all $w$ such that $\sum_{k=1}^{K}w_k\le Q$, $Q$ is an even number, and  each $w_k~\mbox{\rm is even}$; and $\sigma_{i,w}$'s are all nonnegative. It is easy to verify that Assumption \ref{ass:problem} is satisfied. Note that the requirements of $Q$ being an even number and $\sigma_{i,w}$ being nonnegative is to ensure that $f(\by)$ is lower bounded.

 Another example is related to neural network training problem, with SoftPlus activation function $ \phi(x) = \ln(1 + e^x) $, a common smoothed approximation to ReLU. Consider the following problem: \[ f_i(W_1,W_2) = \| W_2 {{\bphi}}(W_1 z_i) - v_i \|^2 \]
where $\left( z_i\in \mathbb{R}^k, v_i \in \mathbb{R} \right)$ is the $i$th data point; {
$ W_2 \in \mathbb{R}^{1 \times h} $ and $ W_1 \in \mathbb{R}^{h \times k} $;} $\bphi(W_1z_i):= [\phi([W_1z_i]_1, \cdots, [W_1z_i]_h)]^T$. If we calculate the gradient with respect to $ W_2 $ and $ W_1 $, we have 
	\begin{align*}
		\nabla_{W_2} f_i(W_1,W_2) &= 2\left( W_2\bphi(W_1z_i) - v_i \right)\bphi(W_1z_i)^T \\
		\nabla_{W_1} f_i(W_1,W_2 ) &= 2(W_2 D)^T (W_2 \bphi( W_1 z_i) - v_i)z_i^T
	\end{align*}
	where $D := \nabla_{h} \bphi(h)|_{h=W_1 z_i} = \mbox{diag}( [\phi^\prime([W_1 z_i]_1), \cdots,  \phi^\prime([W_1 z_i]_h)])$ is a diagonal matrix and each element on its diagonal is between 0 and 1 due to the property of SoftPlus. Furthermore, we have 
	\begin{subequations}\label{nn:example:layer}
		\begin{align}
		\nabla_{W_2} f_i(W_1,W_2 ) - \nabla_{W_2} f_i(W_1,\widetilde{W_2}) &= 2(W_2 - \widetilde{W_2})\bphi(W_1z_i)\bphi(W_1z_i)^T  \label{nn:example:layer2} \\
		\nabla_{W_1} f_i(W_1, W_2) - \nabla_{W_1} f_i(\widetilde{W_1}, W_2) &= 2(W_2D)^T W_2\left(\bphi(W_1z_i) - \bphi(\widetilde{W_1}z_i)\right)z_i^T. \label{nn:example:layer1}
		\end{align}
	\end{subequations}
	It is easy to see that the lhs  of the two equations above are not Lipschitz continuous if $ W_1 $ and $ W_2 $ are unbounded. For \eqref{nn:example:layer2}, the local Lipschitz constant is easy to calculate as long as $ W_1 $ is defined on a compact set. Moreover, since the SoftPlus activation function has Lipschitz constant equal to 1, it is also straightforward to compute the Lipschitz constants for \eqref{nn:example:layer1}. 

\subsection{The MAGENTA Algorithm}
Now we are ready to discuss the proposed algorithm. Our proposed algorithm consists of multiple {\it stages}, each containing multiple {\it inner iterations}. The key idea is to impose a series of artificial constraints $x_i\in X(t)$ to each of the local problem, one for each stage $t$, so that the iterates always stay within some compact sets. At each stage, the following subproblem is considered, which is a usual distributed optimization problem with per-node constraints: 
\begin{align}\label{eq:problem:constrained}
\min\; \frac{1}{N}\sum_{i=1}^{N} f_i(\bx_i), \quad \mbox{s.t.}~{\bx_i = \bx_j}, \; \mbox{\rm if~$(i,j)$~ are neighbors}, \; \bx_i\in X(t), \; \forall~i.
\end{align}
The agents will perform a version of distributed gradient tracking algorithm to compute certain (sufficiently accurate) solutions, and then move on to the next stage. Overall, the key in the design and analysis is to show that, the entire process will produce a desirable solution after a finite number of stages. Further, it is crucial to understand the tradeoff between the following quantities: 1) the accuracy of each stage; 2) the speed in which the constraint set increases; 3)  the total number of iterations required to reach certain $\epsilon$-stationary solution.

Before we state the algorithm, we make the following assumption on the weight matrix. 
\begin{assumption}\label{ass:algorithm}
	{\bf (About the Algorithm).}
	\begin{itemize}
		\item We assume that the weight matrix $\bW$ satisfies \eqref{eq:W}, and it follows that:
		\begin{align}\label{eq:eta}
		\underline{\lambda}_{\max}(\bW) = \eta<1, \; 	\; 
		\|\bW - \frac{1}{N}\b1 \b1^T\|<1,
		\end{align}
		where $\underline{\lambda}_{\max}(\cdot)$ denotes the second largest eigenvalue of $\bW$. 
	\end{itemize}
\end{assumption}

\begin{algorithm}[h]
	\caption{{\small {\it \underline{M}ulti-st\underline{a}ge \underline{g}radi\underline{en}t \underline{t}r{a}cking \underline{a}lgorithm} (MAGENTA) }}
	\label{alg:p1}
	\begin{algorithmic}	
		\State {\bfseries Input:} $\bx{(0)}$, $\by_i(0) = \nabla f_i(\bx_i{(0)}), \; \forall~i$, let $\bz=\bar{\bx}{(0)}$
		\For{$t=0,1,\ldots, T$}
		\State Set $X(t) = B(v^t, \bz)\subset\mathbb{R}^K$ 
		\State Determine the total inner iteration number $R(t)$ 
		\For{$r=0,1,\ldots$}
		\State Let $\bx_i^0 = \bx_i(t)$, $\by_i^0 = \by_i(t)$,   for all $i\in[N]$
		\State Calculate the local Lipschitz constant $L_i(X(t))$, for all $i\in[N]$
		\State Calculate $\hat{L}(X(t))$, for all $i\in[N]$
		\State Calculate $\alpha(t)$ according to \eqref{eq:alpha:choice:2}, set $\beta=1/2$
		\State Calculate the local gradient $\nabla f_i(\bx^r_i)$ for all $i\in[N]$
		\State Update $\bv^{r+1}_i$ by \eqref{eq:vupdate1}  for all $i\in[N]$
		\State Update $\bx^{r+1}_i$ by \eqref{eq:xupdate} for all $i\in[N]$
		\State Update $\by^{r+1}_i$ by \eqref{eq:yupdate} for all $i\in[N]$
		\If{$r+1 = R(t)$}
		\State Output, let  $\bx(t+1) = \bx^{R(t)}$, $\by(t+1) = \by^{R(t)}$
		\EndIf
		\EndFor
		\EndFor
	\end{algorithmic}
\end{algorithm}

{The proposed algorithm, named  \underline{M}ulti-st\underline{a}ge \underline{g}radi\underline{en}t \underline{t}r{a}cking \underline{a}lgorithm (MAGENTA), is given in the table above.}
In the following, each stage is indexed by $t$, while each {\it inner iteration} is indexed by $r$.  In the MAGENTA, $\alpha(t)>0$ and $\beta\in(0,1)$ are the stepsizes to be chosen, and their values can change across different stages; The specific steps of the $\bx_i, \bv_i$ and $\by_i$ updates are given below (where ${\cal{N}}_i$ indicates the neighbor set for node $i$, $\bW_{ij}$ denotes $(i,j)$th element of $\bW$):
\begin{subequations}\label{eq.updateofalg}
	\begin{align}
	\widetilde{\bx}^{r+1}_i & = \arg\min_{x_i\in X(t)} \langle \by^r_i, \bx_i-\bx^r_i\rangle + \frac{1}{2\alpha(t)}\|\bx_i-\bx^r_i\|^2, 
	\quad \bv_i^{r+1}  = \tbx^{r+1}_i - \bx^r_i \label{eq:vupdate1} \\
	\bx^{r+1}_i & =\sum_{j\in\mathcal{N}_i}\bW_{ij}(\bx^r_j + \beta \bv_j^{r+1}), \label{eq:xupdate}
	\\
	\by^{r+1}_i & = \sum_{j\in\mathcal{N}_i}\bW_{ij}\by^r_j+\nabla f_i(\bx^{r+1}_i)-\nabla f_i(\bx^{r}_i). \label{eq:yupdate}
	\end{align}
\end{subequations}
Note that in step \eqref{eq:vupdate1}, each node performs a gradient projection step, where the feasible set $X(t)$ is {\it common} to all the agents, and it is determined in the current outer loop. Also note that Assumption \ref{ass:algorithm} on the $\bW$ matrix ensures that, as long as $\bx^0$ is feasible and $\beta\in(0,1), \; \forall~i$, then $\tbx_i^{r+1}, \bx_i^{r+1}\in X(t)$, for all $r$. 
Additionally, the updates \eqref{eq:vupdate1} -- \eqref{eq:yupdate} can be written compactly as: 
\begin{align}
\bv^{r+1} & = \tbx^{r+1}- \bx^r,\;  \bx^{r+1}= \bW(\bx^r + \beta\bv^{r+1}), \; \by^{r+1}= \bW \by^{r} + \nabla g(\bx^{r+1}) - \nabla g(\bx^r).\label{eq:yupdate:compact}
\end{align}
It is worth noting that in each inner iteration \eqref{eq:vupdate1} -- \eqref{eq:yupdate} , $\mathcal{O}(1)$ local gradient computation and $\mathcal{O}(1)$ local computation take place for each agent. It is also important to note that the MAGENTA can be viewed as a {\it meta algorithm}, in which some existing algorithms for dealing with constrained consensus problem are wrapped around by the outer stages. In Algorithm 1, the inner iterations \eqref{eq.updateofalg} take the form of constrained gradient tracking algorithm \cite{sun2019convergence}, so the analysis below is also based on this algorithm.

%

\begin{remark}
	The sequence of feasible sets  $X(t)$'s is a sequence of balls (centered at origin) with increasing radii. At each given iteration $t$, all the nodes have the common constraint $X(t)$. Such a construction is the key to our proposed algorithm, for the following reasons:
	
\noindent{\bf 1)} For a given stage, when the agents share the common constraint $X(t)$, the analysis can be performed relatively easily; To the best of our knowledge, there has been no existing  analysis for cases where the agents have different constraints (except for the dual perturbation algorithm \cite{Hajinezhad17inexact}, which yields slower rates);

\noindent{\bf 2)}  Because the rules for setting the constraints are simple, the agents can adopt these rules locally without communicating with the neighbors;

\noindent{\bf 3)}  We will see subsequently, that the sequence  $\{v^t\}$ has to be carefully designed to optimize the overall convergence speed.
\end{remark}

\begin{remark}
We note that to compute the stepsize $\alpha(t)$, the agents are required to compute $\hat{L}(t):=\max_{i} L_i(t)$. This operation can be done with in $N$ iterations by performing the classical max-consensus algorithm; see, e.g., \cite{Nejad}. Further, it is noting that requiring the knowledge of the maximum Lipschitz constants is common in the existing algorithms, such as Prox-PDA and GT (see, e.g., Eqs. (D4) -- (D6) in \cite{sun2019convergence}).
\end{remark}

\section{Convergence Analysis}

In this subsection, we analyze the convergence of the two proposed algorithms. 

\subsection{Stationarity Conditions}\label{sub:optimality}
Before we start, we will define the notion of first-order stationarity, as well as the complexity measures we are interested in.  For simplicity, we ignore the stage index $t$ whenever possible. 

Let us start by considering the original unconstrained problem \eqref{eq:problem:main}. Suppose we have an iterate $\bx^r$ available, then it achieves the {\it exact}  first-order stationarity if the following holds:
\begin{align}\label{eq:exact:1}
\left\|\overline{\nabla g(\bx^r)}\right\|^2 + \|\bx^r -\b1 \bar{\bx}^r\|^2=0,
\end{align} 
where $\bar{\bx}^r$ is the average of $\bx^r_1,\cdots, \bx^r_N$ as defined in the notation section, and $\overline{\nabla g(\bx^r)}$ is defined similarly. 
That is, all the local variables are in consensus, and the average of the local gradients is zero. 
Now suppose we have an iterate $\bx^r$ and $\by^r$ generated by the proposed algorithm. Then the algorithm achieves the exact  first-order stationarity if the following {\it stationarity gap} function is zero:
\begin{align}\label{eq:exact:2}
\mbox{gap}(r):=\|\by^r\|^2 + \|\bx^r -\b1 \bar{\bx}^r\|^2 + \|\by^r-\b1\bar{\by}^r\|^2=0.
\end{align} 
To see this, note that from \eqref{eq:yupdate:compact}, the following hold:
\begin{align*}
\bby^{r+1} = \bby^r + \overline{\nabla g(\bx^{r+1})}- \overline{\nabla g(\bx^{r})}.
\end{align*}
By expanding the above equation, and using the initial condition that $\bby^0 = \overline{\nabla g(\bx^{0})}$, we obtain
\begin{align}\label{eq:yg}
\bby^{r} = \overline{\nabla g(\bx^{r})}, \; \forall~r.
\end{align}
Therefore, it is clear that when $\by^r-\b1\bar{\by}^r =0$, the  left-hand-side (lhs) of \eqref{eq:exact:2} is the same as that of \eqref{eq:exact:1}. It follows that \eqref{eq:exact:2} is also a valid notion of exact first-order stationarity for both algorithms. Next, we discuss the $\epsilon$-stationary solution for the proposed algorithms.  

Following \eqref{eq:exact:2} and  by taking into consideration the inner iterations, we can also define the $\epsilon$-stationary solution as follows:
\begin{align}\label{eq:gap:unconstrained}
\mbox{avg-gap}(t) := \frac{1}{R(t)}\sum_{r=0}^{R(t)-1} \left(\|\by^r\|^2 + \|\bx^r -\b1 \bar{\bx}^r\|^2 + \|\by^r-\b1\bar{\by}^r\|^2\right)\le \epsilon.
\end{align} 
Compared with \eqref{eq:exact:2}, we have relaxed the right-hand-side (rhs) from $0$ to some constant $\epsilon>0$, while in the lhs we take an average of the gap function from in $R(t)$ inner iterations. If \eqref{eq:gap:unconstrained} holds true, then there must exist an iteration $r\in [0, R(t)-1]$ such that  $\mbox{gap}(r)\le \epsilon.$
Further, note that at each stage $t$, we are solving a {\it constrained} problem, so a related definition of $\epsilon$- stationary solution is:
\begin{align}\label{eq:gap:constrained}
\widehat{\mbox{avg-gap}}(t):=&\frac{1}{R(t)}\sum_{r=0}^{R(t)-1}\|\bv^{r+1}/\alpha(t)\|^2 +\|\bx^r-\b1 \bbx^r\|^2 + \|\by^r -\b1\bby^r\|^2\le \epsilon.
\end{align}
If $\tbx^{r+1}_i$ in \eqref{eq:vupdate1} {\it does not} touch the boundary of $X(t)$ for all $i$ and all $r\in[0, R(t)-1]$, then  \eqref{eq:gap:constrained} and \eqref{eq:gap:unconstrained} are equivalent. This is because in this case problem \eqref{eq:vupdate1} can be effectively written as an unconstrained problem, and we have, $\bv^{r+1}/\alpha(t) = (\tbx^{r+1}-\bx^r)/\alpha(t) = \by^{r},$
where the last equality comes from the optimality condition \eqref{eq:vupdate1}. On the contrary, if for  some $i$, and some $r$,  $\tbx^{r+1}_i$ touches the boundary of $X(t)$, then \eqref{eq:gap:constrained} and \eqref{eq:gap:unconstrained} are {\it not} equivalent.

\subsection{Convergence Analysis for MAGENTA}\label{sub:convergence:1}

Our analysis below will be focused on characterizing the communication and computation complexities of MAGENTA. That is, the total number of times that the communication steps \eqref{eq:yupdate:compact} are performed for each node, as well as the total number of times that the local gradients $\nabla f_i(\bx_i)$'s are computed by each node, before an $\epsilon$-stationarity solution \eqref{eq:gap:unconstrained} is found.

Let us first focus on the inner iterations of the algorithm. 
Let us define $\bw :=(\bx,\by, \bbx,\bby)$, and define a potential function $P(\cdot)$ as below: {\small
	\begin{align}\label{eq:potential}
	P(\bw^{r+1};t): & =  f(\bbx^{r+1}) + \|\bx^{r+1}-\b1\bbx^{r+1}\|^2 + \frac{1-(1+\gamma)\eta^2}{32 (1+1/\xi)\hat{L}^2(t)}\|\by^{r+1}-\b1\bby^{r+1}\|^2,
	\end{align}}
where we define $\hat{L}(t):=\max_{i} L_i(t)$. 
Clearly, we have $P(\bw^{r+1};t)>\underline{f}$, for all $r$ and $t$, as long as $1-(1+\gamma)\eta^2>0$. 
Then we have the following convergence result, which essentially says that the iterates converge to some $\epsilon$-stationary solutions for the constrained problem \eqref{eq:problem:constrained}. The proof can be found in Appendix \ref{app:proof:T1}.
\begin{theorem}\label{thm:1}
	Suppose that Assumptions \ref{ass:problem} -- \ref{ass:algorithm} hold true. Let $\epsilon>0$ be some constant. Suppose that the algorithm parameters $\gamma, \xi, \beta, \alpha(t)$ are chosen such that the following conditions hold:
	{\small \begin{align}
		& (1+\gamma)\eta^2<1, \quad (1+\xi)\eta^2<1, \quad \beta=\frac{1}{2}\label{eq:alpha:choice:1}\\
		& \alpha(t)= \min\left\{\frac{1}{{8} N}\left(\frac{1}{\hat{L}(t)/2N+(1/\gamma+5/4)}\right),  \frac{N(1-(1+\gamma)\eta^2)(1-(1+\xi)\eta^2)}{{64} (1+1/\xi)\hat{L}^2(t)},1 \right\}\label{eq:alpha:choice:2}.
		\end{align}	
		Then fix stage $t$, assume that the total inner iteration $R(t)$ satisfies $R(t)= \frac{1}{\epsilon\alpha(t)}$, we will have
		\begin{align}\label{eq:epsilon:stationarity}
		{\hspace{-0.3cm}\frac{1}{R(t)}\sum_{r=0}^{R(t)-1}\hspace{-0.2cm}\left({c_1}\left\|\frac{\bv^{r+1}}{\alpha(t)}\right\|^2 + {c_2}\|\bx^r-\b1 \bbx^r\|^2 + \frac{2c_3}{N(1-(1+\gamma)\eta^2)}\|\by^r -\b1\bby^r\|^2\right)\le (P(\bw^0,t)-\underline{f})\epsilon},
		\end{align}}
	where the positive constants $c_1,c_2,c_3$ are given by
	{\small\begin{subequations}\label{eq:c}
			\begin{align}
			& c_1 : = {\frac{1}{4}\left(\frac{1}{4N} - \frac{L(t)\alpha(t)}{2N} - \alpha(t)\left(\frac{1}{\gamma}+\frac{5}{4}\right)\right)>
				\frac{1}{32 N}>0\label{eq:c1},}\\
			& c_2: = \frac{1}{2}\left((1-(1+\gamma)\eta^2)-\frac{\hat{L}^2(t)\alpha(t)}{N}\right)>\frac{(1-(1+\gamma)\eta^2)}{4}>0,\\
			&  c_3 = \frac{(1-(1+\gamma)\eta^2)(1-(1+\xi)\eta^2)}{{32} (1+1/\xi)} - {\frac{\alpha(t) \hat{L}^2(t)}{N}} \geq {\frac{(1-(1+\gamma)\eta^2)(1-(1+\xi)\eta^2)}{ 64 (1+1/\xi)} >0}.
			\end{align}
	\end{subequations}}
\end{theorem}

\begin{remark}
	Let us comment on the choice of the parameters. First, from Assumption 2 we have $\eta<1$, therefore it is always possible to choose $\gamma$ and $\xi$ (independent of $t$) such that  \eqref{eq:alpha:choice:1} holds true. Then fixing $\gamma$,  $\eta$, $\hat{L}(t)$ and $L(t)$, we can find $\alpha(t)$ such that \eqref{eq:alpha:choice:2} holds true. The key relation in  \eqref{eq:alpha:choice:2} is that the stepsize $\alpha(t)$ is inversely proportional to $\hat{L}^2(t)$, and $\hL(t)$. Such a dependency will be critical in obtaining the overall convergence rate estimate. 
	
	Further, by using the choices in \eqref{eq:c}, the lhs of \eqref{eq:epsilon:stationarity} can be further lower bounded as:
	\begin{align*}
	&\frac{1}{R(t)}\sum_{r=0}^{R(t)-1}\frac{1}{{32N}} \|\bv^{r+1}/\alpha(t)\|^2 +\frac{(1-(1+\gamma)\eta^2)}{4}\|\bx^r-\b1 \bbx^r\|^2 + \frac{1-(1+\xi)\eta^2}{{32N}(1+1/\xi)}\|\by^r -\b1\bby^r\|^2.
	\end{align*}
	It is clear that the constants in front of the  three terms in the summation  are not dependent on $t$, therefore if we use the $\epsilon$-{stationarity gap} function defined in \eqref{eq:gap:constrained}, then the above inequality implies that
	\begin{align*}
	\widehat{\mbox{avg-gap}}(t)=&\frac{1}{R(t)}\sum_{r=0}^{R(t)-1}\|\bv^{r+1}/\alpha(t)\|^2 +\|\bx^r-\b1 \bbx^r\|^2 + \|\by^r -\b1\bby^r\|^2 = \mathcal{O}(\epsilon).
	\end{align*}
	This means that after running $R(t)$ iterations that satisfies $R(t) = 1/(\epsilon \times \alpha(t))$, we have reached some $\epsilon$-stationary solution for the constrained problem.
	\hfill $\blacksquare$
\end{remark}

\begin{remark}
	There have been existing analyses on applying GT to constrained consensus problem (with common constraints among the agents) \cite[Theorem 2.16]{sun2019convergence}, and they can essentially be adapted to apply to analyze the inner iterations. However, there are a number of places that we need to be careful when directly adopting these results. First, the  dependency on various problem parameters such as $L(t)$ and $L(t)^2$ have to be specified in the construction of the potential function, since later the potential functions in different stages have to be stitched together. Perhaps most importantly, the dependency of the convergence rates of the inner iteration with problem parameters such as $\hat{L}$ has to be sharpened, since these dependencies will play a critical role in finding the overall bound. Due to the above reasons, we choose to conduct a separate analysis.  
	\hfill $\blacksquare$
\end{remark}

Despite the fact that Theorem \ref{thm:1} suggests that for each stage we have $\widehat{\mbox{avg-gap}}(t) = \mathcal{O}(\epsilon)$, it is not the final result yet. The reason is that, as we have mentioned at the end of Sec. \ref{sub:optimality}, the stationarity gap measure $\widehat{\mbox{avg-gap}}(t)$ is not the same as the desired ${\mbox{avg-gap}}(t)$, unless $\tbx^{r+1}_i$ computed in \eqref{eq:vupdate1} {\it does not} touch the boundary of $X(t)$ for all $i$ and all $r\in[0, R(t)-1]$.  
Fortunately, the following result shows that, there exists an upper bound $T^*$ such that, for some  $t\in [1, T^*]$, the iteration generated in stage $t$ will not touch the boundary of $X(t)$.

\begin{theorem}\label{thm:2}
	Suppose that the assumptions made in Theorem \ref{thm:1} hold true. 
	Let $\epsilon>0$ be a fixed constant satisfying $\epsilon< 2 (1 - (1+\gamma)\eta^2)$. Suppose ${T}^*>0$ satisfies the follows:
	\begin{align}
	{T}^* & = \left\lceil \frac{1}{\epsilon}\frac{128  ({P}(\bw^0;1)-\underline{f})}{{\frac{1}{{T}^*}\sum_{t=1}^{{T}^*}(v^{t}-v^{t-1})^2}}\right\rceil.\label{eq:hatT}
	\end{align}
	Then if MAGENTA is executed by ${T}^*$ stages, there must exist $t\in [1, {T}^*]$ such that  $\{\bx^r_1(t), \cdots, \bx^r_N(t)\}_{r=1}^{R(t)}$ and $\{\tbx^r_1(t), \cdots, \tbx^r_N(t)\}_{r=1}^{R(t)}$ do not touch the boundary of $X(t)$, that is
	\begin{align}
	\mbox{\rm Dist}(\bx^r_i(t),X(t))>0, \; \mbox{\rm Dist}(\tbx^r_i(t),X(t))>0, \; \; \forall~i, \mbox{and}~r\in [0,R(t)].
	\end{align}
\end{theorem}

Combining Theorems \ref{thm:1} --  \ref{thm:2}, we arrive at the conclusion that if one can find $T^*$ that satisfies \eqref{eq:hatT}, then there must exist a stage $t\in [1, {T}^*]$ such that $\mbox{avg-gap}(t)=\mathcal{O}(\epsilon)$. The final step is to estimate the total communication and computation complexity required to achieve $\epsilon$-stationarity \eqref{eq:gap:unconstrained}. The result below is a straightforward combination of Theorems \ref{thm:1} -- \ref{thm:2}, 
\begin{theorem} \label{thm:3}
	Suppose that the assumptions made in Theorem \ref{thm:1} hold true. For a given $\epsilon>0$ satisfying $\epsilon< 2 (1 - (1+\gamma)\eta^2) $, to achieve $\epsilon$-stationary solution as defined in \eqref{eq:gap:unconstrained}, MAGENTA requires the following rounds of local gradient computation and message exchanges among the neighbors (where $T^*$ satisfies \eqref{eq:hatT}):
	\begin{align}\label{eq:total:complexity}
	\frac{1}{\epsilon} \sum_{t = 1}^{T^*} \frac{1}{\alpha(t)} = \frac{1}{\epsilon}\sum_{t=1}^{T^*}{\max\left\{4\hat{L}(t) + 8N/\gamma+10N, \frac{64 (1+1/\xi)\hat{L}^2(t)}{N(1-(1+\gamma)\eta^2)(1-(1+\xi)\eta^2)},1 \right\}}.
	\end{align}
\end{theorem}

At this point, how the above estimate depends on $\epsilon$ is not immediately clear. 
The precise dependency relies on algorithm parameters (e.g., how the sequence $\{v^t\}$ is chosen) and the functional class of each local function $f_i$'s  (which determines the Lipschitz constants $\hat{L}(t)$ and $L(t)$). The following result shows how to make specific choices of these parameters to yield a meaningful complexity bound, as well as how to specialize such a bound to a specific problem.

\hspace{-0.3cm}
\begin{Corollary}\label{eq:cor:poly}
	Suppose that the assumptions made in Theorem \ref{thm:1} hold true. Assume $v^t = t\times d$ for some $d>0$. Then to achieve $\epsilon$-stationary solution as defined in \eqref{eq:gap:unconstrained}, MAGENTA requires the following rounds of communication and local computation:
	\begin{align}\label{eq:complexity}
	\mathcal{O}\left(\left\lceil \frac{128  ({P}(\bw^0;1)-\underline{f})}{\epsilon{d^2}}\right\rceil  \frac{1}{\epsilon} \mbox{\rm gr}\left(d \times \left\lceil \frac{128 ({P}(\bw^0;1)-\underline{f})}{\epsilon{d^2}}\right\rceil \right)\right),
	\end{align}
	where we have defined an increasing growth function $\mbox{\rm gr}(v):=\max\{L(\mathbb{B}(v,0)), \hat{L}^2(\mathbb{B}(v,0))\}$.

	If additionally each $f_i(\bx)$ is a $Q$th order polynomial given below (with $Q\ge 2$):
	\begin{align}\label{eq:poly}
	f_i(\bx) = \sum_{w:\sum_{k=1}^{K}w_k\le Q} \sigma_{i,w} x^{w_1}_1\times \cdots \times x^{w_K}_K,
	\end{align}
	where $w =[w_1;\cdots; w_K]$, with $w_k$'s satisfying $w_k\ge 0, \;\forall~k$, $\sum_{k=1}^{K}w_k\le Q$; $\sigma_{i,w}$'s are coefficients for the polynomial. By setting $d= 1/\sqrt{\epsilon}$, the complexity in \eqref{eq:complexity} becomes  $O(1/\epsilon^{Q-1})$. 
\end{Corollary}

\begin{remark} \label{discussion:cor}
	In the proof of the above result, the sequence $\{v^t\}$ is chosen as $v^t= d\times t$ where $d>0$ is some constant. We showed that when $d$ is a constant that is {\it independent of $\epsilon$}, the complexity is in the order of $\mathcal{O}(1/\epsilon^{2Q-2})$, while when $d=1/\sqrt{\epsilon}$, the order becomes $\mathcal{O}(1/\epsilon^{Q-1})$. Obviously, the latter is much better than the former, and such a comparison illustrates our previous point that the precise complexity result depends on a number of problem and algorithmic parameters, thus it should be analyzed on a case-by-case basis. 
	
	It is also important to note that for quadratic problems with $Q=2$, both the LLC and GLC hold.
	In this case, $\mathcal{O}(1/\epsilon^{Q-1})$ becomes $\mathcal{O}(1/\epsilon)$, and such a rate is {\it tight} because it matches the known rate bounds for decentralized non-convex optimization; see a summary of such results in Table \ref{fig:table_compare} and a recent survey \cite{chang2020distributed}. At this point, an open question is whether the rate derived above is tight for generic $Q$th order polynomial and more general non-convex problems satisfying Assumptions \ref{ass:problem} -- \ref{ass:algorithm}. \hfill $\blacksquare$
\end{remark}

\section{Numerical Results}\label{sec:numerics}
In this section, we present results for a few simple numerical experiments. 

\begin{figure}[!b]
	\minipage{0.33\textwidth}
	\includegraphics[width=\linewidth]{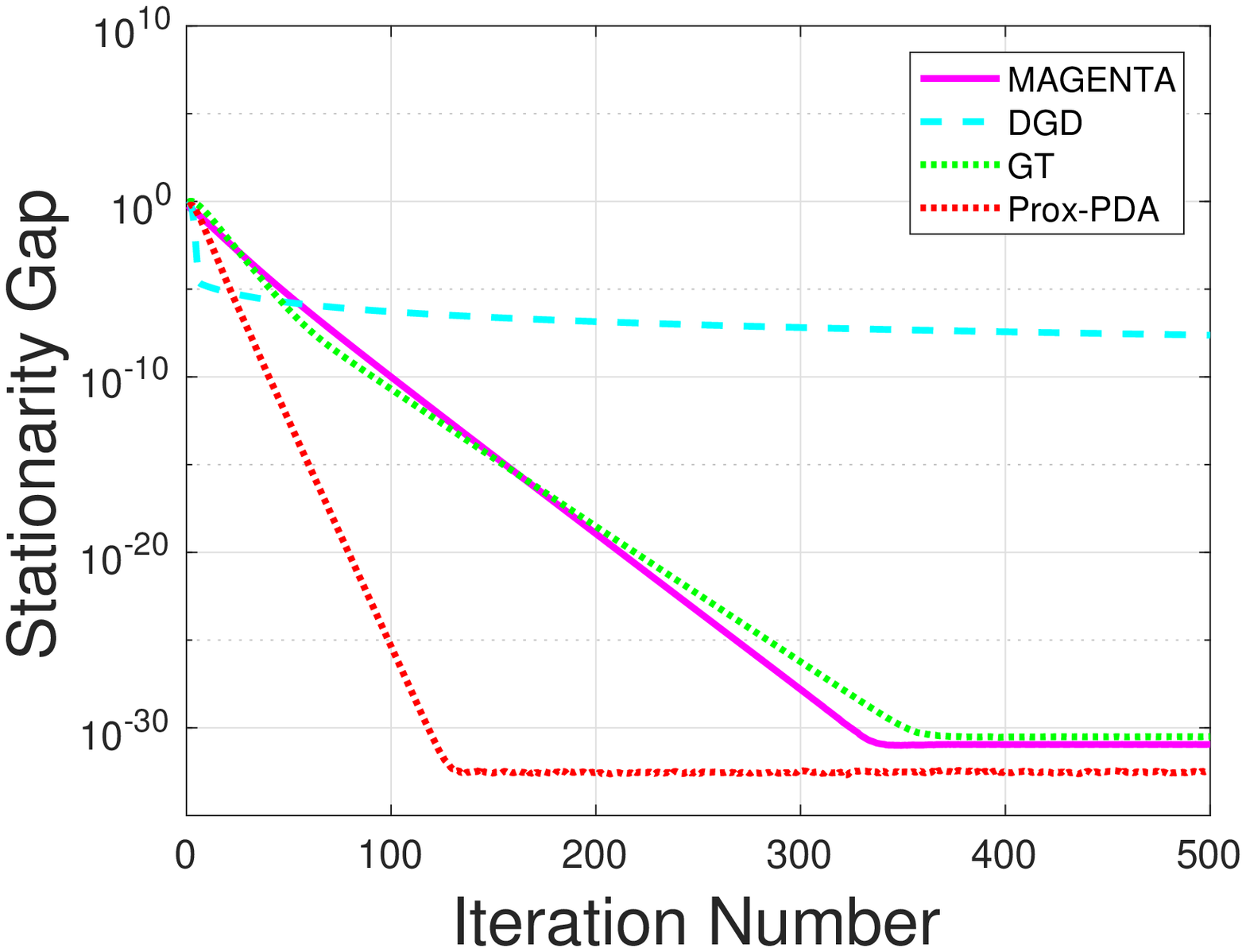}
	\endminipage\hfill
	\minipage{0.33\textwidth}
	\includegraphics[width=\linewidth]{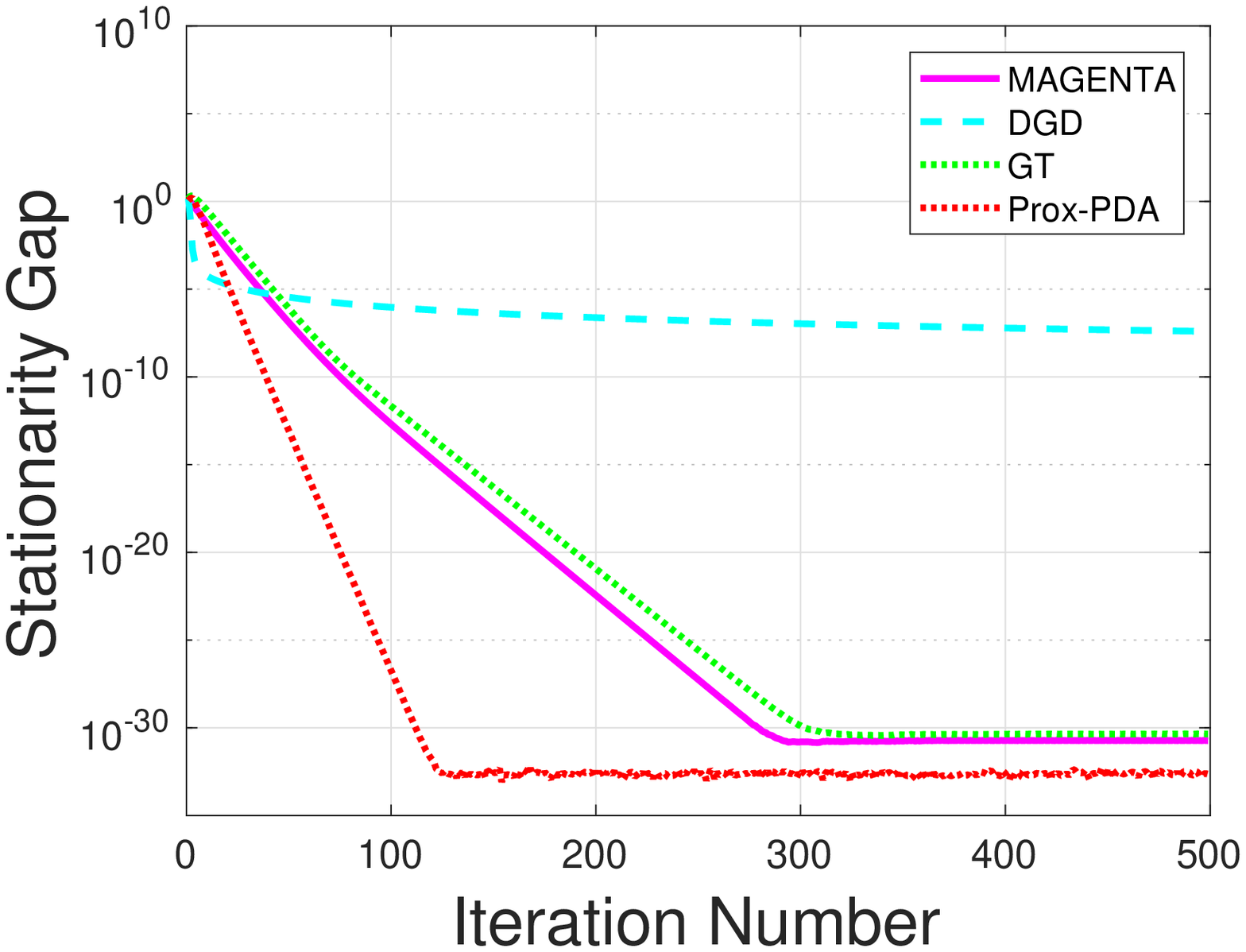}
	\endminipage\hfill
	\minipage{0.33\textwidth}%
	\includegraphics[width=\linewidth]{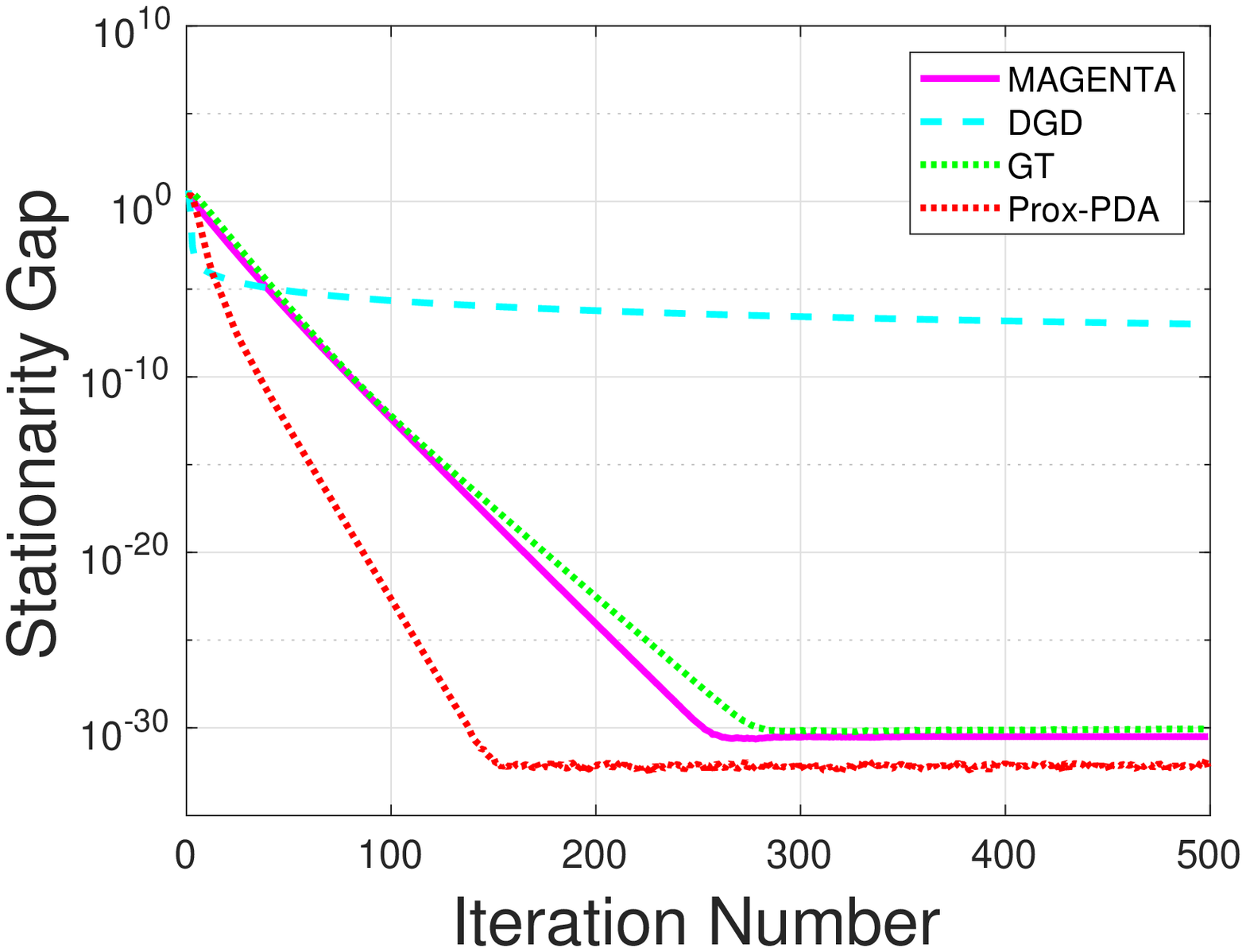}
	\endminipage
	\caption{\small Results for the logistic regression problem on several graph topologies. (Left) $n = 5$ agents and $ M_i = 400 $ data points on each agent; (Middle) $n = 10$ agents and $M_i = 200$ data points on each agent; (Right) $n = 20$ agents and $M_i = 100$ data points on each agent. The iteration number here represents the inner iteration number index, in which each node performs one round of local gradient update.} 
	\label{result:logistic_node}
\end{figure}

{\bf Experiment Set I:} First, apply MAGENTA to a typical decentralized problem which satisfies both GLC and LLC. We consider a regularized logistic regression problem with a non-convex regularizer in a distributed manner \cite{antoniadis2011penalized}. With $ \lambda, \rho > 0 $ be the regularizer's parameters, each local cost function $ f_i $ is expressed as follows:
\begin{align}\label{loss:logistic}
f_i(\bm \theta_i) = \frac{1}{M_i}\sum_{l = 1}^{M_i}\log \left( 1 + \exp(-b_i^l \bm{\theta}_i^T \bm{a}_i^l)\right) + \lambda \sum_{s = 1}^{d} \frac{\rho \theta_{i, s}^2}{1 + \rho \theta_{i, s}^2}
\end{align}
where $ \bm a_i^l $ and $ b_i^l $ are the features and the label of the $i$-th data point, and $ \bm{\theta}_i$'s are the optimization variables. Here, we set the problem dimension to be $ K = 5 $ and generate $ M = 2,000 $ data points. We equally distribute the data on each agent, where we choose $n\in \{5, 10, 20\}$ in our experiments. For each setting the graph is generated using the random geometric graph, where we place the nodes uniformly in $[0, 1]^2$ and connect any two nodes separated by a distance less than a radius $R \in (0, 1)$. The graph parameter $R$ is set to 0.5 in our simulation. We note that one can verify that problem \eqref{loss:logistic} satisfies the GLC and LLC, therefore, this experiment serves as a sanity check of the proposed algorithm. Moreover, it is also easy to calculate the local Lipschitz constant over a compact region for the gradient of the problem \eqref{loss:logistic}.

In Fig \ref{result:logistic_node}, we show the stationarity gap (as defined in the lhs of \eqref{eq:exact:1}) versus the iteration number of the Prox-GPDA \eqref{eq:pd}, Gradient Tracking \eqref{eq:gt} and DGD \eqref{eq:dgd}. In this figure, each line is the average of 5 independent runs. For the MAGENTA, we set $\epsilon=10^{-4}$ and $v^t = t\times 1/\sqrt{\epsilon}$. Moreover, the stepsize of the rest of the algorithms are also well-tuned so that they can achieve their best performance. It is clear that MAGENTA is able to quickly reduce the stationarity gap and it achieves comparable performance with GT and Prox-PDA.

{\bf Experiment Set II:} Next, we consider a simple polynomial optimization problem which satisfies neither GLC \eqref{eq:Lip:central} nor the LLC \eqref{eq:Lip:distributed}. We consider a simple example where the sum function is $ f(u) = (u-20)^4 $ and component functions are $ f_1(x_1) = \frac{1}{2} \left( x_1 - 20 \right)^4 $, $ f_2(x_2) = \frac{1}{2} \left( x_2 - 20 \right)^4 $. Then we compare performance of different algorithms.

\begin{table}[!t]
	\vspace{-2px}
	\caption{Convergence of  decentralized algorithms on a polynomial optimization problem} 
	\label{fig:experiment_summarary}
	\footnotesize
	\begin{center}
		\begin{sc}
			\begin{tabular}{lcccccccc}
				\toprule
				Algorithm  & $c=1$ & $c=1/2$ & $c=1/4$ & $c=1/8$  \\
				\midrule
				DGD \cite{zeng2018nonconvex} & $ 12.5\% $  & $ 62.3\% $ & $ 72.3\% $ & $ 96.8\% $ \\
				Gradient Tracking \cite{nedic2017achieving} & $ 13.1\% $  & $ 66.5\% $ & $ 92.0\% $ & $ 98.1\% $ \\
				Prox-PDA \cite{hong2017prox}& $ 51.1\% $  & $ 72.9\% $ & $ 85.8\% $ & $ 94.5\% $ \\
				\bottomrule
			\end{tabular}
		\end{sc}
	\end{center}
	\vspace{-1cm}
\end{table}

\begin{figure}[!b]
	\minipage{0.45\textwidth}
	\includegraphics[width=\linewidth]{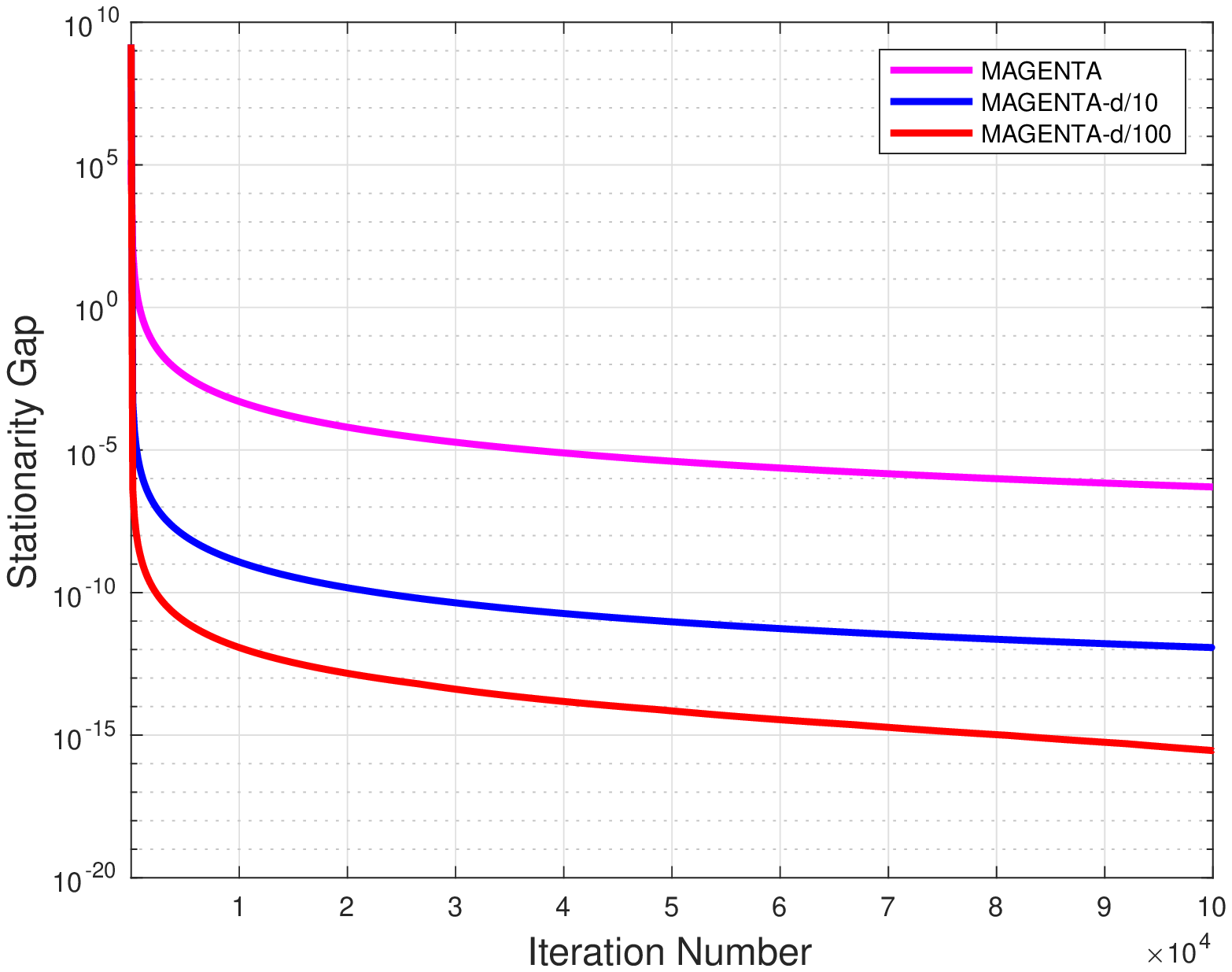}
	\endminipage\hfill
	\minipage{0.45\textwidth}
	\includegraphics[width=\linewidth]{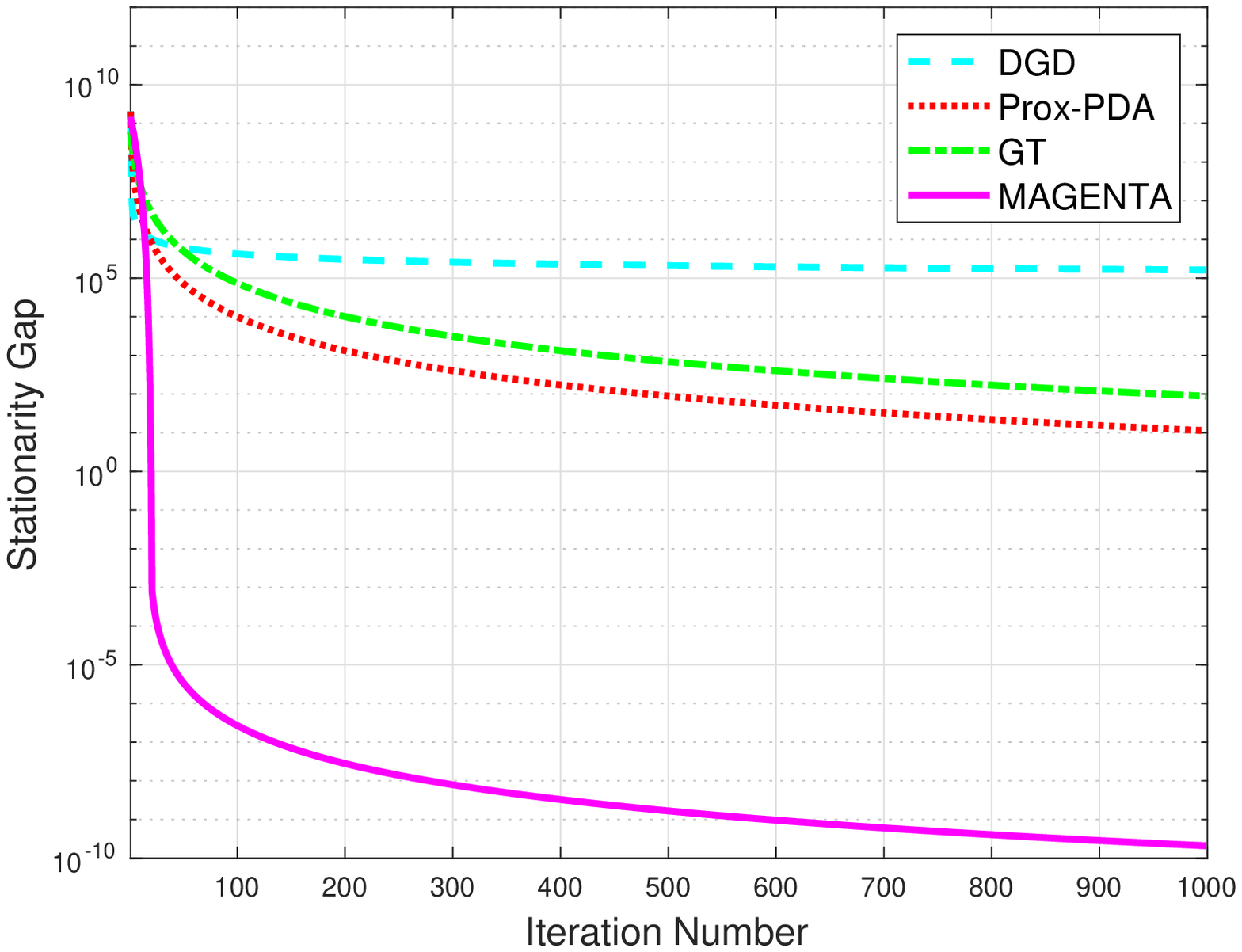}
	\endminipage
	\caption{\small Experiment Set II: a polynomial optimization problem. (Left) The performance of MAGENTA with different radius $ v^t = t \times d $
		; (Right) The comparison of different decentralized algorithms.  Each line is an average of $20$ independent runs.} 
	\label{result:polynomial}
\end{figure}

In the experiments, the initial solutions are generated from the standard Gaussian distribution. We note that for the polynomial optimization problem the existing decentralized algorithms (such as DGD, GT, etc.) cannot guarantee convergence, therefore we do not have appropriate rules to choose the stepsizes. So we try the following heuristic:  choose the stepsizes inversely proportional to the norm of the initial solutions, i.e.,  $\frac{c}{\|\bx^0\|}$, and then find the largest possible $c$ so that these algorithms can converge. 
Specifically, for DGD \cite{zeng2018nonconvex}, we set the stepsize sequence  as $\{ \alpha_r \}_{r = 1}^\infty = \{ \frac{5c \times 10^{-3}}{r \times \|\bx^0\| }  \}_{r = 1}^\infty $; for GT \cite{nedic2017achieving}, we set $ \alpha = \frac{2c \times 10^{-3}}{\|\bx^0\|}$; for Prox-PDA, we set the penalty parameter as $ \rho = \frac{10^3 \times \|\bx_0\|}{c}  $.
In Table \ref{fig:experiment_summarary}, we show the percentage of the convergent instances for each benchmark algorithms (out of $1,000$ runs). In each run, we say that a given algorithm converges if the stationarity gap reduces to below $ 10^{0} $, and we regard it as divergent if the stationarity gap exceeds $ 10^{10}$.

In the Fig. \ref{result:polynomial} (Left), we show the performance of MAGENTA with different radius sequences. Here, we set $ \epsilon = 10^{-4} $ and $ d = 1/\sqrt{\epsilon} = 100 $. Moreover, the stepsize in each stage is tuned to be inversely proportional $(\hat{L}(t))^2$ (cf. \eqref{eq:alpha:choice:2}). We also provide three kinds of radius sequence $ \{v^t\} $ to be $ v^t = t \times d = 100t $, $ v^t = t \times \frac{d}{10} = 10t $ and $  v^t = t \times \frac{d}{100} = t$. 
From these results, we observe that the performance of MAGENTA is improved when the radius is smaller,  since the allowed stepsizes are larger. 
In Fig. \ref{result:polynomial} (Right), we compare the convergence speed for each algorithm. For the benchmark algorithms (DGD, GT and Prox-PDA), we choose $c=1/4$ in order to guarantee convergence of most of the runs while ensuring that the stepsizes are as large as possible (see the third column in Table \ref{fig:experiment_summarary}). In MAGENTA, its target stationary gap $ \epsilon $ is set to be $ \epsilon = 10^{-4} $ and the radius sequence $ \{v^t\} $ is $ v^t = t \times \frac{1}{100\sqrt{\epsilon}} = t $. The stepsize of MAGENTA in each stage is also set to be inversely proportional to $ (\hat{L}(t))^2$.  Clearly, MAGENTA is the fastest  among all the tested algorithms.

{\bf Experiment Set III:} In this experiment set, we apply different algorithms to train a simple neural network. We consider a regression problem with $ N = 4 $ agents and there are $ n_i = 100$ data points in each agent $ i $. For each data point $ \{z_i, v_i\} $, the feature  $ z_i $ is a 3-dimensional vector and its label $ v_i $ is a scalar. In each agent, there is a three-layer neural network with one hidden layer and the size of the neural network is $ 3 \times 5 \times 1 $. Therefore, the component function in each agent could be expressed as follows:\begin{align}\label{nn:distributed}
f_i(\bm W^i) &=  \frac{1}{n_i}\sum_{k = 1}^{n_i} \| W_2^i \bphi(W_1^i z_k) - v_k \|^2 
\end{align}
where $ \phi(\cdot) $ denotes the ReLU activation function and $ \bm{W^i} = \{W_1^i, W_2^i\} $ is the optimization variable of the neural network in the $ i $th agent. {The dimension of the optimization variables $ W_1^i, W_2^i $ are $ 5\times3 $ and $ 1 \times 5 $.}

In our experiments, the initial weights of the neural networks are generated from Gaussian distribution and we evaluate the results for training neural networks with different initial weights. According to \eqref{nn:example:layer}, we have shown that the neural network does not  satisfy {LLC} \eqref{eq:Lip:distributed} and {GLC} \eqref{eq:Lip:central}. 
For this reason, we again follow the previous heuristic, and choose the  stepsizes for the benchmark algorithms according to the initial solutions. In Table \ref{fig:neural_network_summary}, we show the percentage of the converging instances for the benchmark algorithms (out of $100$ runs), when changing the constant $c$. 

\begin{table}[!t]
	\vspace{-1cm}
	\caption{Convergence of  decentralized algorithms for training neural network with initial weights generated from standard normal distribution.}
	\label{fig:neural_network_summary}
	\footnotesize
	\begin{center}
		\begin{sc}
			\begin{tabular}{lcccccccc}
				\toprule
				Algorithm  & $c=1$ & $c=1/2$ & $c=1/4$ & $c=1/8$  \\
				\midrule
				DGD \cite{zeng2018nonconvex} & $ 0\% $  & $ 12\% $ & $ 99\% $ & $ 100\% $ \\
				Gradient Tracking \cite{nedic2017achieving} & $ 0\% $  & $ 46\% $ & $ 99\% $ & $ 100\% $ \\
				Prox-PDA \cite{hong2017prox}& $ 0\% $  & $ 83\% $ & $ 100\% $ & $ 100\% $ \\
				\bottomrule
			\end{tabular}
		\end{sc}
	\end{center}
\end{table}

\begin{table}[h]
	\vspace{-2px}
	\caption{Convergence of  decentralized algorithms for training neural network with initial weights generated from normal distribution $ \mathcal{N}(0, 10) $.}
	\label{fig:neural_network_summary_2}
	\footnotesize
	\begin{center}
		\begin{sc}
			\begin{tabular}{lcccccccc}
				\toprule
				Algorithm  & $c=1$ & $c=1/2$ & $c=1/4$ & $c=1/8$  \\
				\midrule
				DGD \cite{zeng2018nonconvex} & $ 4\% $  & $ 53\% $ & $ 100\% $ & $ 100\% $ \\
				Gradient Tracking \cite{nedic2017achieving} & $ 1\% $  & $ 54\% $ & $ 100\% $ & $ 100\% $ \\
				Prox-PDA \cite{hong2017prox}& $ 0\% $  & $ 73\% $ & $ 100\% $ & $ 100\% $ \\
				\bottomrule
			\end{tabular}
		\end{sc}
	\end{center}
\end{table}

In  Fig. \ref{result:nn_training}, we compare the convergence speed of each algorithm when training the neural network described above. For the benchmark algorithms (DGD, GT and Prox-PDA), we choose $c=1/4$ in order to guarantee convergence and ensure that the stepsizes are as large as possible (see the third column in Table \ref{fig:neural_network_summary} and Table \ref{fig:neural_network_summary_2}). 
{For MAGENTA, we choose $ \epsilon = 10^{-2} $, $ d = \frac{1}{10\sqrt{\epsilon}} = 1 $ and a diminishing stepsize sequence, which is proportional to $1/\sqrt{t}$ (where $t$ is the stage index). Note that in this case, we are not exactly following the theoretical stepsize choices. However, the general guidelines provided in \eqref{eq:alpha:choice:2}, that the stepsizes should be decreasing with $t$, is still useful}. From this figure, we can see that MAGENTA can quickly decrease the stationarity gap, and it is the fastest among the benchmark algorithms. 

\begin{figure}[!t]
	\vspace{-0.5cm}
	\minipage{0.45\textwidth}
	\includegraphics[width=\linewidth]{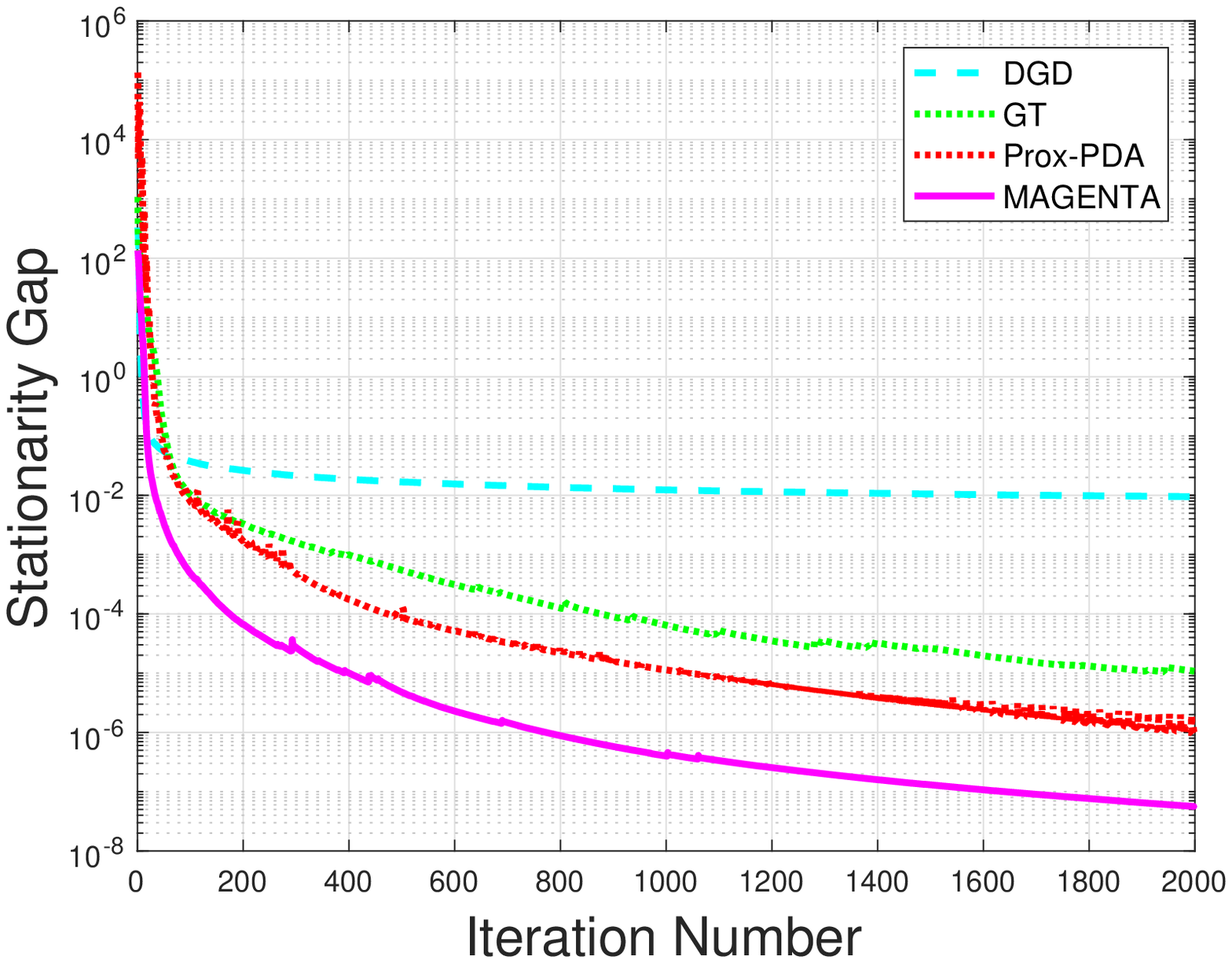}
	\endminipage\hfill
	\minipage{0.45\textwidth}
	\includegraphics[width=\linewidth]{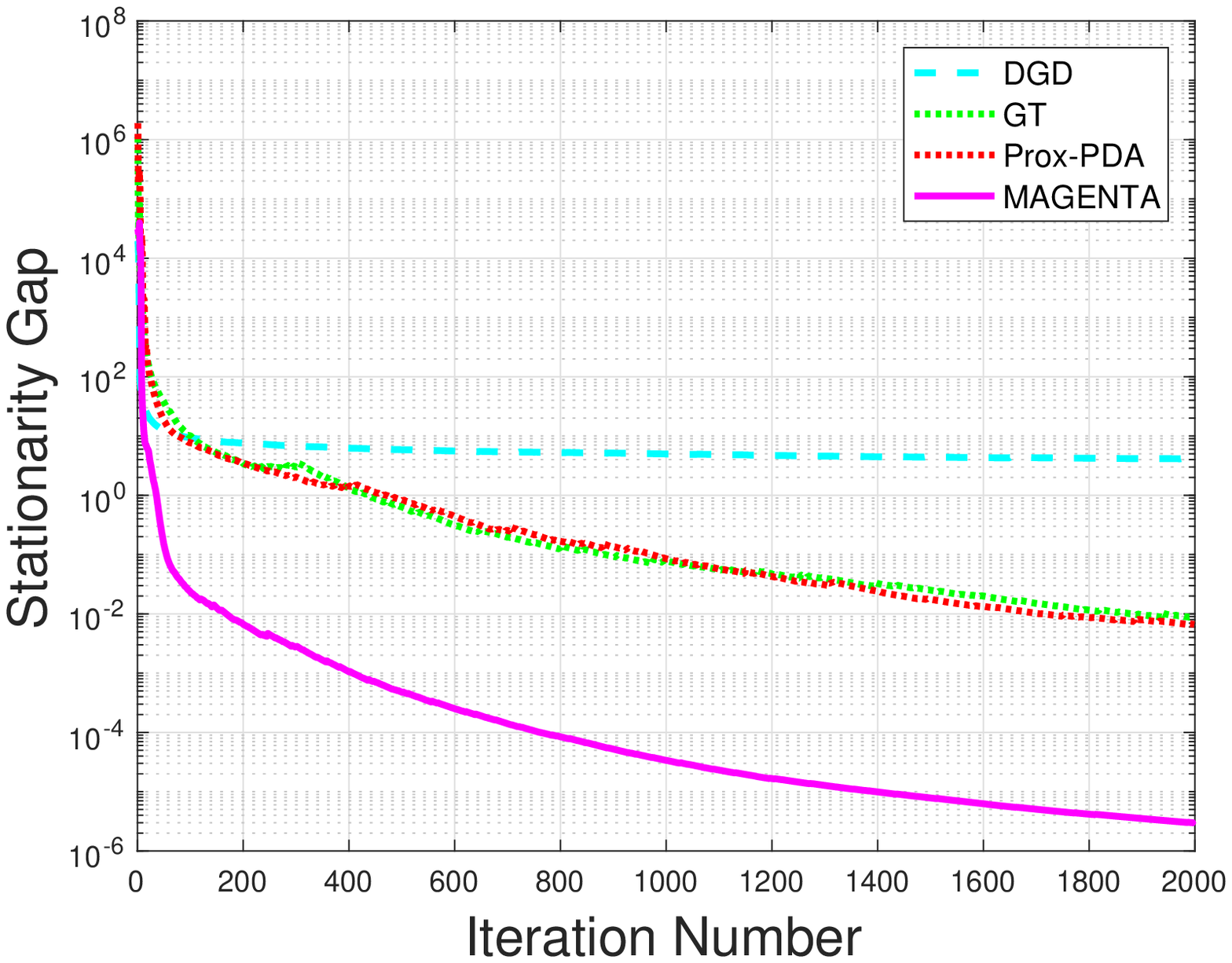}
	\endminipage
	\caption{\small Experiment Set III: Decentralized optimization for training a three-layer neural network. Each line contains an average of $10$ independent runs of the respective algorithm. (Left) The initial weights are generated from standard normal distribution; (Right) The initial weights are generated from $ \mathcal{N}(0, 10) $.} 
	\label{result:nn_training}
	\vspace{-0.4cm}
\end{figure}

\section{Concluding Remarks}
In this work, we provide in-depth understanding about decentralized optimization problems. We show that existing algorithms are critically dependent on certain Lipschitz gradient assumption on the local and the global objective functions. We then provide a novel scheme, and the accompanying analysis, which essentially removes the requirements of these assumptions. We analyze the total local gradient and communication complexity for the proposed methods, and specialize these results to a polynomial optimization problem. We expect that our approach can be extended to other first-order methods, so that they can also work without the need to have global Lipschitz constant.
Finally, our work poses the following fundamental open question: Without assuming the LLC, or even the GLC, what is the {\it best} achievable convergence rate for the class of first-order decentralized optimization algorithms?

\newpage

{\small
	\bibliographystyle{IEEEtran}
	\bibliography{ref_VR,ref_Decentralized,ref,ref_Decentralized_2020}}

\clearpage
\newpage
\appendix

\section{Proofs of Section \ref{sec:global_local}}

Let us define some notations first. Since the inner problem is solved with constraint $\bu\in X$, let us introduce the indicator function $h(\cdot): \mathbb{R}^{K}\to \mathbb{R}$ for this constraint, that is $h(\bu)=1$ if $\bu\in X$ and $h(\bu)=\infty$ otherwise. Then problems \eqref{eq:problem:central} and \eqref{eq:problem:main} can be equivalently written as
\begin{align}
\min & \; \frac{1}{N}\sum_{i=1}^{N} (f_i(\bu) + h(\bu)) = f(\bu) + h(\bu): = F(\bu)\\
\min & \; \frac{1}{N}\sum_{i=1}^{N} (f_i(\bx_i) + h(\bx_i))  = g(\bx) + \ell(\bx): = G(\bx), \quad \mbox{s.t.}~{\bx_i = \bx_j}, \; \mbox{\rm if~$(i,j)$~ are neighbors}, \label{eq:problem:main:constrained}
\end{align}
where $\ell(\bx):=\frac{1}{N}\sum_{i=1}^{N} h(\bx_i)$.

\subsection{Proof of Lemma \ref{lemma:equivalence}}
\begin{proof}\label{proof:equivalence}
	It is straightforward to show that LLC implies GLC. Further, the LLC implies that 
	\begin{align*}
	\|\nabla f(\by)- \nabla f(\bz)\|\le \frac{1}{N}\sum_{i=1}^{N}\|\nabla f_i(\by)- \nabla f_i(\bz)\|\stackrel{\eqref{eq:Lip:distributed}}\le \frac{1}{N}\sum_{i=1}^{N} L_i\|\by-\bz\|, \quad \forall~\bz, \by \in \mbox{dom}(f)
	\end{align*}
	where the first inequality comes from the Jensen's inequality. 
	Therefore we will show that for convex problems, if GLC holds then LLC holds. 
	
	First, let us assume that each $f_i$ is second-order differentiable. This condition will be removed shortly. 
	Let $H_i(\bx)\in\mathbb{R}^{K\times K}$ and $H(\bx)\in\mathbb{R}^{K\times K}$ denote the Hessian matrix for $f_i(\cdot)$, and for $f(\cdot)$, respectively. Since each $f_i$ is convex and second-order differentiable, then its Hessian matrix satisfies: $H_i(\bx)\succeq 0$. For functions in this class, we also know that  the LLC \eqref{eq:Lip:central} is equivalent to the following condition
	\begin{align*}
	\lambda_{1}(H_i(\bx))\le L_i, \; \forall~i, \; \forall~\bx,
	\end{align*}
	Similarly, if the GLC holds, then $\lambda_{1}(H(\bx))\le L, \; \forall~\bx$. Suppose that GLC holds but LLC does not hold. Then it follows that there exists $i\in[N]$ such that 
	\begin{align*}
	\lambda_{1}(H_i(\by))\ge 2 N L, \; \mbox{for some $\by$}.
	\end{align*}
	It follows that for some $j\in[N]$ and $j\ne i$, 
	\begin{align*}
	\lambda_{1}(H(\by))\ge \frac{1}{N}\lambda_{1}(H_i(\by))\ge 2L.
	\end{align*}
	This contradicts to the fact that $\lambda_{1}(H(\bx))\le L, \; \forall~\bx$. The proof is completed.

	Second, assume that each $f_i$ may not be second-order differentiable, so the Hessian matrices $H_i(\bx)$'s may not exist. From \cite[Theorem 2.1.5]{Nesterov04}, it is known that the GLC is equivalent to the following
	\begin{align}\label{eq:contradiction}
	f(\bx) \le f(\by) + \langle \nabla f(\by), \bx-\by\rangle + \frac{L}{2}\|\bx-\by\|^2,\; \forall~\by,\bx,
	\end{align}
	and similar equivalence holds for each component function $f_i$'s, under the $LLC$. 
	
	Suppose that there exists $j\in[N]$ such that LLC \eqref{eq:Lip:distributed} does not hold. Then for {\it any} $L'>0$, there exists $\bx,\by$ such that the following condition {\it does not} hold. 
	\begin{align*}
	f_j(\bx) \le  f_j(\by) + \langle \nabla f_j(\by), \bx-\by\rangle + \frac{L'}{2}\|\bx-\by\|^2, \; \mbox{for all $\bx, \by$}.
	\end{align*}
	It follows that for a given index $j\in[N]$,  there exists $\tbx,\tby$ such that
	\begin{align}\label{eq:j}
	f_j(\tbx) \ge f_j(\tby) + \langle \nabla f_j(\tby), \tbx-\tby\rangle +  {N}{L}\|\tbx-\tby\|^2, \; \mbox{for some $\tbx, \tby$}.
	\end{align}
	Since all $f_i$'s are convex, we have, for the same pair of vectors $\tbx, \tby$, the following holds
	\begin{align}\label{eq:i}
	f_i(\tbx) \ge f_i(\tby) + \langle \nabla f_i(\tby), \tbx-\tby\rangle, \; \forall~i\in[N].
	\end{align}
	Adding the realtions in \eqref{eq:i} for all $i\ne j$, then add the relation \eqref{eq:j}, and dividing both sides by $N$, we obtain
	\begin{align*}
	\frac{1}{N}\sum_{i=1}^{N} f_i(\tbx) & \ge  \frac{1}{N}\sum_{i=1}^{N} f_i(\tby) + \frac{1}{N}\left\langle \nabla \left(\sum_{i=1}^{N} f_i(\tby)\right), \tbx-\tby\right\rangle + L\|\tbx-\tby\|^2 \nonumber\\
	& =  f(\tby) + \left\langle \nabla f(\tby), \tbx-\tby\right\rangle + {L}\|\tbx-\tby\|^2, \; \mbox{for some $\tbx, \tby$}.
	\end{align*}
	This is a contradiction to \eqref{eq:contradiction}. The desired result is proved.
\end{proof}

\subsection{Proof of Claim \ref{claim:dgd}}
\begin{proof} \label{proof:dgd}
	Consider problem \eqref{eq:problem:main} with $ K= 1$ (scalar variables), and assume that we have  $N= 2$ agents connected via one edge. Assume that the mixing matrix is chosen as  $\displaystyle \bm W = \begin{bmatrix}
	1/2 & 1/2 \\
	1/2 & 1/2
	\end{bmatrix}$. Let us first assume that the DGD iteration uses a constant stepsize $\alpha>0$. 
	
	Let us assume the following:
	\begin{align}\label{eq:counter:1}
	f_1(x_1) = \frac{1}{3} x_1^3, \quad  f_2(x_2) = -\frac{1}{3}x_2^3.
	\end{align}
	Clearly the LLC is not satisfied, while the GLC  is, because $ f_1(y) + f_2(y) = 0 $. It is also easy to observe that, any solutions satisfying $x_1=x_2$ are optimal for the decentralized  problem \eqref{eq:problem:main}. 
	
	Let us then specialize the DGD using the above choices of the parameters, and use it to solve problem specified in \eqref{eq:counter:1}.  We will have the following iteration: 
	\[ \bx^{r+1} = \bm W\bx^r - \alpha \nabla g(\bx^r) =  \b1 \times \bar{x}^r - \alpha \times \begin{bmatrix}
	(x^r_{1})^2 \\ 
	-(x^r_{2})^2\\
	\end{bmatrix}, \]
	where the $g(\cdot)$ function is defined in \eqref{eq:problem:main}, and $\nabla g(\bx):=[\nabla f_1(x_1); \nabla f_2(x_2)].$
	
	For a given $r\ge 0$, let us assume that $x^r_{2}-x^r_{1} \geq \frac{2}{\alpha} $ and $\bar{x}^r > 0 $, then we have
	\begin{align}
	&(x^{r+1}_{2} - x^{r+1}_{1}) - (x^r_{2} - x^r_{1}) \nonumber\\
	&= \alpha\left((x^r_{2})^2 + (x^r_{1})^2\right)- (x^r_{2} - x^r_{1})\nonumber\\
	&= \alpha\left(2(\bar{x}^{r})^2 + \frac{(x^r_{2}-x^r_{1})^2}{2} - \frac{1}{\alpha}(x^r_{2} - x^r_{1})\right) \nonumber\\
	&= \alpha\left(2(\bar{x}^{r})^2 + \frac{(x^r_{2}-x^r_{1})}{2}(x^r_{2}-x^r_{1} - \frac{2}{\alpha}) \right) \nonumber\\
	&\geq 2\alpha (\bar{x}^r)^2. \label{eq:difference:large}
	\end{align}
	For the sequence $ \{\bar{x}^r\}_{r= 0}^\infty $, the following holds true
	\begin{align}\label{eq:average:large}
	\bar{x}^{r+1} &= \frac{1}{2} \b1^T \left(\bm W\bx^r - \alpha \nabla g(\bx^r)\right) = \bar{x}_k - \frac{\alpha}{2}\b1^T \nabla g(\bx^k)\nonumber\\
	\bar{x}^{r+1} - \bar{x}^r & = \frac{\alpha}{2}(x^r_{2} - x^r_{1})(x^r_{2} + x^r_{1}) > 0.	
	\end{align}
	Summarizing \eqref{eq:difference:large} and \eqref{eq:average:large}, and suppose that the initial solution $\bx^0$ satisfies 
	\begin{align}\label{eq:initial}
	\bar{x}^0 > 0, \quad  x^0_{2} - x^0_{1} \geq \frac{2}{\alpha}.
	\end{align}
	we then obtain:
	\begin{align*}
	x^{r+1}_{2}-x^{r+1}_{1} \geq \frac{2}{\alpha}, \;\; x^{r+1}_{2}-x^{r+1}_{1} > x^r_{2} - x^r_{1} \geq \frac{2}{\alpha}, \;\;  \bar{x}^{r+1} > \bar{x}^r > 0, \; \forall~r\ge 0.
	\end{align*}
	
	\[  x^{r+1}_{2} - x^{r+1}_{1} \geq (x^0_{2} - x^0_{1}) + 2\alpha \sum_{t= 0}^{r}(\bar{x}^t)^2 \geq \frac{2}{\alpha} + 2\alpha (r+1) (\bar{x}^0)^2.\]
	In summary, for any given positive $\alpha>0$, there exists an initial solution $\bx^0$ satisfying \eqref{eq:initial}, so that
	the sequence $ \{\bar{x}^r\}^\infty_{r = 0}$ as well as the consensus error $\{x^r_{2} - x^r_{1}\}_{r = 0}^\infty$ diverges to infinity. 
	This completes the first part of the claim. 
	
	{Then let us assume that $\alpha^r$ is a diminishing sequence satisfying \eqref{eq:stepsize}.} Now we specialize the DGD algorithm and consider the decreasing stepsize $ \alpha^r := \frac{\alpha}{r+1} > 0$. For a given $r\ge 0$, let us assume that $x^r_{2}-x^r_{1} \geq \frac{2}{\alpha^r} = \frac{2(r+1)}{\alpha} $ and $\bar{x}^r > 0 $, then we have
	\begin{align}
	&(x^{r+1}_{2} - x^{r+1}_{1}) - (x^r_{2} - x^r_{1}) \nonumber\\
	&= \alpha^r \left((x^r_{2})^2 + (x^r_{1})^2\right)- (x^r_{2} - x^r_{1})\nonumber\\
	&= \alpha^r \left(2(\bar{x}^{r})^2 + \frac{(x^r_{2}-x^r_{1})^2}{2} - \frac{1}{\alpha^r}(x^r_{2} - x^r_{1})\right) \nonumber\\
	&= \alpha^r \left(2(\bar{x}^{r})^2 + \frac{(x^r_{2}-x^r_{1})}{2}(x^r_{2}-x^r_{1} - \frac{2}{\alpha^r}) \right) \nonumber\\
	&\geq 2\alpha^r (\bar{x}^r)^2. \label{eq:difference:large2}
	\end{align}
	
	For the sequence $ \{\bar{x}^r\}_{r= 0}^\infty $, the following holds true
	\begin{align}\label{eq:average:large2}
	\bar{x}^{r+1} &= \frac{1}{2} \b1^T \left(\bm W\bx^r - \alpha^r \nabla g(\bx^r)\right) = \bar{x}_k - \frac{\alpha^r}{2}\b1^T \nabla g(\bx^r)\nonumber\\
	\bar{x}^{r+1} - \bar{x}^r & = \frac{\alpha^r}{2}(x^r_{2} - x^r_{1})(x^r_{2} + x^r_{1}) > 0.	
	\end{align}
	
	Let us further assume $ \alpha^r (\bar{x}^r)^2 \geq \frac{1}{\alpha} $, then it is obtained that 
	\begin{equation}
	(x^{r+1}_{2} - x^{r+1}_{1}) \geq (x^r_{2} - x^r_{1}) + 2\alpha^r (\bar{x}^r)^2 \geq \frac{2(r+1)}{\alpha} + \frac{2}{\alpha} = \frac{2(r+2)}{\alpha} = \frac{2}{\alpha^{r+1}} \label{eq:dgd:proveassumption1}
	\end{equation}
	Furthermore, we can show that 
	\begin{align}
	(\bar{x}^{r+1})^2 &= \left( \bar{x}^r + \alpha^r \bar{x}^r (x^r_{2} - x^r_{1}) \right)^2 \nonumber \\
	&= (\bar{x}^r)^2 + (\alpha^r)^2 (\bar{x}^r)^2 (x^r_{2} - x^r_{1})^2 + 2\alpha^r (\bar{x}^r)^2(x^r_{2} - x^r_{1}) \nonumber\\
	&\geq \frac{r+1}{\alpha^2} + \frac{4(r+1)}{\alpha^2} + \frac{4(r+1)}{\alpha^2} \nonumber \\
	&= \frac{9(r+1)}{\alpha^2} \label{eq:dgd:proveassumption2}.
	\end{align}
	Summarizing \eqref{eq:difference:large2}--\eqref{eq:dgd:proveassumption2}, and suppose that the initial solution $ \bm x^0 $ satisfies 
	\begin{align}
	x_2^0 - x_1^0 \geq \frac{2}{\alpha^0}, \quad \alpha^0 (\bar{x}^0)^2 \geq \frac{1}{\alpha}, \quad \bar{x}^0 > 0, \label{eq:dgd:initial2}
	\end{align} we obtained that 
	\begin{align*}
	x_2^{r+1} - x^{r+1}_1 \geq \frac{2}{\alpha^{r+1}} = \frac{2(r+2)}{\alpha},\quad \alpha^{r+1} (\bar{x}^{r+1})^2 \geq \frac{1}{\alpha}, \quad \bar{x}^{r+1} > \bar{x}^r >0.
	\end{align*}
	In summary, if we replace the constant stepsize to the diminishing stepsizes, then there exists a sequence that satisfies the condition \eqref{eq:stepsize} and an initial solution $ \bm x^0 \in \mathbb{R}^K$, so that the DGD iteration diverges.
\end{proof}

\subsection{Proof of Claim \ref{claim:gt}}
\begin{proof} \label{proof:gt}
	Consider the problem \eqref{eq:counter:1}. Assume the mixing matrix is chosen as $ \bm W = \begin{bmatrix}
	1/2 & 1/2 \\
	1/2 & 1/2
	\end{bmatrix} $, then we will have the following iteration:
	\begin{align}
	\bm x^{r+1} &= 2\bm{W} \bm x^r - \bm{W}^2 \bm x^{r-1} - \alpha\left(\nabla g(\bm x^r) - \nabla g(\bm x^{r-1})\right) \nonumber \\
	& = \bm 1 \times 2 \bar{x}^r - \bm 1 \times \bar{x}^{r-1} - \alpha\left( \begin{bmatrix}
	(x_1^r)^2\\
	-(x_2^r)^2
	\end{bmatrix} - \begin{bmatrix}
	(x_1^{r-1})^2 \\
	-(x_2^{r-1})^2
	\end{bmatrix} \right) \label{eq:GT:update}.
	\end{align}
	
	For a given $ r \geq 0 $, let us assume that 
	\begin{subequations}\label{eq:condition1}
		\begin{align}
		& x_2^r - x_{1}^r \geq x_{2}^{r-1} - x_{1}^{r-1} \geq x_{2}^{r-2} - x_{1}^{r-2} \geq 0 \\
		& (x_2^r - x_1^r) - (x_2^{r-1} - x_1^{r-1}) - \frac{2}{\alpha} \geq 0.
		\end{align}
	\end{subequations}

	According to \eqref{eq:GT:update}, then we obtain
	\begin{align}
	&(x_{2}^{r+1} - x_{1}^{r+1}) - (x_{2}^r - x_{1}^r) \nonumber\\
	=&\alpha\left( (x_2^r)^2 + (x_1^r)^2 - (x_2^{r-1})^2 - (x_1^{r-1})^2 - \frac{1}{\alpha} (x_{2}^r - x_{1}^r) \right) \nonumber\\
	=& \alpha \left( 2(\bar{x}^r)^2 - 2(\bar{x}^{r-1})^2 + \frac{(x_{2}^r - x_{1}^r)^2}{2} - \frac{(x_{2}^{r-1} - x_{1}^{r-1})^2}{2}- \frac{1}{\alpha} (x_{2}^r - x_{1}^r)  \right) \nonumber\\
	=& \alpha \left( 2(\bar{x}^r)^2 - 2(\bar{x}^{r-1})^2 + \frac{(x_{2}^r - x_{1}^r)}{2}(x_{2}^r - x_{1}^r - \frac{2}{\alpha}) - \frac{(x_{2}^{r-1} - x_{1}^{r-1})^2}{2}  \right) \nonumber\\
	\geq& \alpha \left( 2(\bar{x}^r)^2 - 2(\bar{x}^{r-1})^2 + \frac{(x_{2}^{r-1} - x_{1}^{r-1})}{2}(x_{2}^r - x_{1}^r - \frac{2}{\alpha}) - \frac{(x_{2}^{r-1} - x_{1}^{r-1})^2}{2}  \right) \nonumber\\
	=&\alpha \left(2(\bar{x}^r)^2 - 2(\bar{x}^{r-1})^2 + \frac{(x^{r-1}_{2} - x^{r-1}_{1})}{2}\left((x_{2}^r - x_{1}^r) - (x_{2}^{r-1} - x_{1}^{r-1}) - \frac{2}{\alpha}\right) \right) \nonumber\\
	\geq& 2\alpha\left( (\bar{x}^r)^2 - (\bar{x}^{r-1})^2 \right) \label{eq2:difference:large}
	\end{align} 
	where the first equality is obtained by subtracting $\bm x^r$ and multiplying $ [-1, 1]^T $ on the both sides of \eqref{eq:GT:update}; {the second equality is by the formula $ (x_2^r)^2 + (x_1^r)^2 = \frac{(x_2^r + x_1^r)^2}{2} + \frac{(x_2^r - x_1^r)^2}{2} = 2(\bar{x}^r)^2 + \frac{(x_2^r - x_1^r)^2}{2} \text{ } \forall r$.}
	
	The sequence $ \{ (\bar{x}^r)^2 - (\bar{x}^{r-1})^2 \}_{r=1}^\infty $ can be expressed as $ \{(\bar{x}^r - \bar{x}^{r-1})(\bar{x}^r + \bar{x}^{r-1})\}_{r=1}^\infty $. Moreover, we have
	\begin{align}
	\bar{x}^{r} &= \frac{1}{2}\bm{1}^T \left[2\bm{W} \bm x^{r-1} - \bm{W}^2 \bm x^{r-2} - \alpha\left(\nabla g(\bm x^{r-1}) - \nabla g(\bm x^{r-2})\right) \right] \nonumber \\
	&= 2\bar{x}^{r-1} - \bar{x}^{r-2} - \frac{\alpha}{2}\bm 1^T \left(\nabla g(\bm x^{r-1}) - \nabla g(\bm x^{r-2})\right). \label{eq:GT:mean-difference}
	\end{align}
	Hence, it is obtained from \eqref{eq:GT:mean-difference} so that 
	\begin{align}
	&(\bar{x}^r - \bar{x}^{r-1}) - (\bar{x}^{r-1} - \bar{x}^{r-2}) \nonumber \\
	=& \frac{\alpha}{2}\left( (x_2^{r-1})^2 - (x_1^{r-1})^2 - (x_2^{r-2})^2 + (x_1^{r-2})^2 \right) \nonumber \\
	=&\frac{\alpha}{2}\left( (x_{1}^{r-1}+x_{2}^{r-1})(x_{2}^{r-1} - x_{1}^{r-1}) - (x_{1}^{r-2}+x_{2}^{r-2})(x_{2}^{r-2} - x_{1}^{r-2}) \right) \nonumber \\
	=& \alpha \left( \bar{x}^{r-1}(x_{2}^{r-1} - x_{1}^{r-1}) - \bar{x}^{r-2} (x_{2}^{r-2} - x_{1}^{r-2}) \right) \label{eq2:difference:mean}.
	\end{align}
	Let us further assume 
	\begin{align}\label{eq:condition2}
	\bar{x}^{r-1} \geq \bar{x}^{r-2} \geq 0,  \; (\bar{x}^{r-1})^2 - (\bar{x}^{r-2})^2 \geq \frac{1}{\alpha^2}.
	\end{align}
	Summarizing \eqref{eq2:difference:large} and \eqref{eq2:difference:mean}, we then obtain that, under the conditions \eqref{eq:condition1} and \eqref{eq:condition2}, the following relations hold: 
	\begin{subequations}\label{eq:GT:result}
		\begin{align}
		& (\bar{x}^{r} - \bar{x}^{r-1}) \geq (\bar{x}^{r-1} - \bar{x}^{r-2}) \geq 0 \label{eq:GT:R1}\\
		& (\bar{x}^r)^2 - (\bar{x}^{r-1})^2 \geq (\bar{x}^{r-1})^2 - (\bar{x}^{r-2})^2 \geq \frac{1}{\alpha^2} \label{eq:GT:R2}\\
		& (x_{2}^{r+1} - x_{1}^{r+1}) - (x_{2}^r - x_{1}^r) - \frac{2}{\alpha} \geq 2\alpha((\bar{x}^r)^2 - (\bar{x}^{r-1})^2 - \frac{1}{\alpha^2}) \geq 0. \label{eq:GT:R3}
		\end{align}
	\end{subequations}
	
	Given a initial solution $ \bm x^{-2} $ and the stepsize $ \alpha $, we can generate $ \bm x^{-1}, \bm x^{0}, \bm x^1 $. Here, $ \bm x^{-1} = \bm W \bm x^{-2} - \alpha \nabla g(\bm x^{-2})$ and other iterations follow \eqref{eq:GT:update}. Suppose the initial solutions satisfy (take $ r = 1 $):
	
	\begin{subequations}\label{eq.GT.assumption}
		\begin{align}
		& x_{2}^1 - x_{1}^1 \geq x_{2}^{0} - x_{1}^{0} \geq x_{2}^{-1} - x_{1}^{-1} \geq 0 \label{eq:GT:A1} \\
		& (x_{2}^1 - x_{1}^1) - (x_{2}^{0} - x_{1}^{0}) - \frac{2}{\alpha} \geq 0 \label{eq:GT:A2} \\
		& \bar{x}^{0} \geq \bar{x}^{-1} \geq 0 \label{eq:GT:A3} \\
		& (\bar{x}^{0})^2 - (\bar{x}^{-1})^2 \geq \frac{1}{\alpha^2} \label{eq:GT:A4}.
		\end{align}
	\end{subequations}
	
	We then obtain:
	\begin{align*}
	& x_{2}^{r+1} - x_{1}^{r+1} \geq x_{2}^{r} - x_{1}^{r} \geq x_{2}^{r-1} - x_{1}^{r-1} \geq 0\\
	& (x_{2}^{r+1} - x_{1}^{r+1}) - (x_{2}^{r} - x_{1}^{r}) - \frac{2}{\alpha} \geq 0\\
	& \bar{x}^{r} \geq \bar{x}^{r-1} \geq 0 \\
	& (\bar{x}^{r})^2 - (\bar{x}^{r-1})^2 \geq \frac{1}{\alpha^2}.
	\end{align*}
	
	Given $ \bm x^{-2} = (\frac{1}{\alpha}, \frac{1}{\alpha}) $ and a fixed stepsize $ \alpha $, then it is obtained that $ \bm x^{-1} = (0, \frac{2}{\alpha}), \bm x_0 = (\frac{2}{\alpha}, \frac{4}{\alpha}), \bm x^1 = (\frac{1}{\alpha}, \frac{17}{\alpha}) $. For $r = 1$, the points $ \bm x^{-1}, \bm x^0, \bm x^1 $ satisfy the four assumptions above. The sequence $ \{x_{2}^r - x_{1}^r\}_{r = 1}^\infty $ diverges to infinity since $  x_{2}^{r+1} - x_{1}^{r+1} \geq (x_{2}^1 - x_{1}^1) + \frac{2r}{\alpha} = \frac{2r+16}{\alpha} $.
	
	In summary, for any given positive $\alpha>0$, there exists an initial solution $\bx^0$ satisfying \eqref{eq:GT:A1} -- \eqref{eq:GT:A4}, so that the sequences $ \{\bar{x}^r\}_{r=1}^\infty$ and $\{x^r_{2} - x^r_{1}\}_{r = 1}^\infty$ diverge to infinity. 
\end{proof}

\subsection{Proof of Claim \ref{claim:pd}}
\begin{proof} \label{proof:pd}
	Consider the same problem \eqref{eq:counter:1}. The degree matrix $ D $ is an identity matrix and the graph incidence matrix $ \bm{A}:= [1,\; -1] $. Therefore, the largest eigenvalue of the graph Laplacian matrix $ \bm{A}^T \bm A $ is $ \lambda_{\max}(\bm A^T \bm A) = 2 $. According to the update \eqref{eq:pd} with the penalty parameter $ \rho $ and the regularization parameter $ \beta = 0 $, the iteration of the parameter $ \bm x $ are expressed as follows:
	\[ \bm x^{r+1} = \left( I -\frac{1}{2} \bm{A}^T \bm A \right)(2\bm x^r - \bm x^{r-1}) - \frac{1}{2\rho} \left(\nabla g(\bm x^r) - \nabla g(\bm x^{r-1})\right). \]
	Set $ \alpha = \frac{1}{2\rho} $ and denote $ \bm W := I -\frac{1}{2} \bm{A}^T \bm A = \begin{bmatrix}
	1/2 &1/2\\
	1/2 &1/2
	\end{bmatrix} $, then we can rewrite the expression above as 
	\begin{align*}
		x_{k+1} &= \bm W (2\bm x^r - \bm x^{r-1}) - \alpha \left(\nabla g(\bm x^r) - \nabla g(\bm x^{r-1})\right) \\
		&= 2\bm W \bm x^r - \bm W^2 \bm x^{r-1} - \alpha\left( \nabla g(\bm x^r) - \nabla g(\bm x^{r-1}) \right)
	\end{align*}
	where in the second equality we use $ \bm W = \bm W \bm W $. Then the analysis of this algorithm follows similar steps as that for gradient tracking.
	
	According to the iterations \eqref{eq:pd:iteration1}--\eqref{eq:pd:iteration2} and given the initial solution $ \bm x^{-2} $ at the first step, we can express the update of $ \bm x $ as $$ \bm x^{-1} = (I - \frac{1}{2} \bm A^T \bm A)\bm x^{-2}  - \alpha \nabla g(\bm x^{-2}) = \bm W \bm x^{-2} - \alpha \nabla g(\bm x^{-2}) $$ since the the dual variable $ \bm \lambda^{-2} = 0 $ at the first iteration. Then the Prox-GPDA is equivalent to the gradient tracking for the problem \eqref{eq:counter:1}.
	
	Given $ \bm x^{-2} = (\frac{1}{\alpha}, \frac{1}{\alpha}) $, then it is obtained that $ \bm x^{-1} = (0, \frac{2}{\alpha}), \bm x^0 = (\frac{2}{\alpha}, \frac{4}{\alpha}), \bm x^1 = (\frac{1}{\alpha}, \frac{17}{\alpha}) $. For this case, it is shown that Prox-GPDA also diverges to infinity according to the divergence of gradient tracking.
	
	In summary, for any penalty parameter $\rho >0$, there exists an initial solution, so that the sequence $ \{\bar{x}^r\}_{r=1}^\infty$ as well as the consensus error $\{x^r_{2} - x^r_{1}\}_{r = 1}^\infty$ diverges to infinity. 
\end{proof}

\section{Proofs of Section \ref{sub:convergence:1}}\label{app:proof:T1}

For notational simplicity, let us define $L_i(t):= L_i(X(t)), \; \forall~i$, $L(t):=L(X(t))$, and 
$\hat{L}(t): =\max_i L_i(t).$ \; 
For simplicity, throughout this subsection we will ignore the index $t$ whenever possible. That is, we will use $L_i$ (resp. $L$) to denote $L_i(t)$ (resp. $L(t)$). 

\subsection{Preliminary Results}
To begin our analysis, let us first present a few key properties of the MAGENTA. For a given vector $\{\bx_i\}_{i=1}^{N}$, let us define $\bbx :=\frac{1}{N}\b1^T \bx$, and define $\bbv, \bby, \tbbx$ similarly. 
From the $\bx$-update step and $\bv$-update \eqref{eq:yupdate:compact} we have
$\bbx^{r+1} = \bbx^r +  \beta {\bbv}^{r+1}$. 

By using the update rule of $\bv$ in \eqref{eq:vupdate1}, we have
\begin{align}\label{eq:xxtilde:equal}
\bbx^{r+1} = \bbx^r + \beta{\bbv}^{r+1} =\bbx^r + \beta(\tbbx^{r+1}-\bar{\bx}^r) = (1-\beta)\bar{\bx}^r + \beta \tbbx^{r+1}.
\end{align}
Further, recall that from \eqref{eq:yupdate:compact} we can obtain
\begin{align}\label{eq:ybar}
\bby^{r+1} = \bby^r + \overline{\nabla g(\bx^{r+1})}- \overline{\nabla g(\bx^{r})}.
\end{align}
By iterating the above equation, and use the initial condition that $\bby^0 = \overline{\nabla g(\bx^{0})}$, we obtain:
$$\bby^{r} = \overline{\nabla g(\bx^{r})}, \; \forall~r.$$
Note that this property also ensures, when entering the next stage, say $t+1$, we still have: 
$$\bar{\by}(t+1) = 1/N\sum_{i=1}^{N}\nabla f_i(\bx{(t+1)}) = \overline{\nabla g(\bx(t+1))}.$$
We also have the following relations:
\begin{align}\label{eq:barx:difference}
\|\bbx^{r+1}-\bbx^r\|^2 & =\left\|\frac{\beta}{N}\sum_{i=1}^{N}(\tbx^{r+1}_i - \bx^r_i)\right\|^2 \le \frac{1}{N}\sum_{i=1}^{N}\beta^2\|\tbx^{r+1}_i - \bx^r_i\|^2 =  \frac{\beta^2}{N}\|\tbx^{r+1}-\bx^r\|^2
\end{align}
where we have applied the Jensen's inequality. Also we have:
\begin{align}\label{eq:gradient:y}
\|\nabla f(\bbx^r)-\bby^r\|^2 & \stackrel{\eqref{eq:yg}}= \|\nabla f(\bbx^r) - \overline{\nabla g(\bx^r)}\|^2 = \left\|\frac{1}{N}\sum_{i=1}^{N}(\nabla f_i(\bbx^r)-\nabla f_i(\bx^r_i))\right\|^2\nonumber\\
&\le \frac{1}{N}\sum_{i=1}^{N}L^2_i\|\bbx^r-\bx^r_i\|^2 \le \frac{\hat{L}^2}{N}\|\bx^r-\b1 \bbx^r\|^2.
\end{align}
Additionally, note that
\begin{align}\label{eq:bbv}
\|\bbv^{r+1}\|^2 = \left\|\frac{1}{N}\sum_{i=1}^{N}(\tbx_i^{r+1} - \bx^r_i) \right\|^2\le \frac{1}{N}\|\tbx^{r+1}-\bx^r\|^2 = \frac{1}{N}\|\bv^{r+1}\|^2.
\end{align}

\subsection{Useful Lemmas}
Our first lemma characterizes the descent of the algorithm. 
\begin{lemma} \label{lemma:decreasing}
	Suppose that Assumptions 1-2 hold true. Within a fixed stage $t$, we have the following descent estimate for all $r\ge 0$:
	\begin{align*}
	f(\bbx^{r+1}) + h(\bbx^{r+1}) & \le f(\bbx^{r}) + h(\bbx^r)\nonumber\\
	& \quad - \beta\left(\frac{1}{8\alpha N}  -\frac{L\beta}{2N}\right)  \|\bv^{r+1}\|^2 + \frac{\hat{L}^2 \alpha \beta}{N} \|\bx^r-\b1 \bbx^r\|^2+ \frac{2 \alpha \beta }{N} \|\by^r -\b1\bby^r\|^2.
	\end{align*}
\end{lemma}
\begin{proof} \label{proof:decreasing_lemma}
	First, by the Lipschitz condition \eqref{eq:f:lip}, we have
	\begin{align}
	f(\bbx^{r+1})& \le f(\bbx^{r}) + \langle \nabla f(\bbx^{r}), \bbx^{r+1}-\bbx^r\rangle + \frac{L}{2}\|\bbx^{r+1}-\bbx^r\|^2\nonumber\\
	f(\bbx^{r+1}) + h(\bbx^{r+1}) & \le f(\bbx^{r}) + h(\bbx^r)+ \underbrace{\langle \nabla f(\bbx^{r}), \bbx^{r+1}-\bbx^r\rangle}_{\rm term A}  + \frac{L}{2}\|\bbx^{r+1}-\bbx^r\|^2 + \underbrace{h(\bbx^{r+1})-h(\bbx^r)}_{\rm term B},\label{eq:descent}
	\end{align}
	where the second inequality is true because we know that $\bx_i^r$'s and $\bx_i^{r+1}$'s are all feasible for $X$, so their averages are also feasible. 
	Let us then analyze term A and term B in \eqref{eq:descent}. We have the following
	\begin{align}\label{eq:A+B}
	{\rm Term A} + {\rm Term B} & \stackrel{\eqref{eq:xxtilde:equal}}   =  h((1-\beta)\bbx^r+\beta\tbbx^{r+1})-h(\bbx^r) + \langle \nabla f(\bbx^{r}), \bbx^{r+1}-\bbx^r\rangle\nonumber\\
	& \le  \frac{\beta}{N}\sum_{i=1}^{N}(h(\tbx^{r+1}_i) - h(\bx_i^r)) +  \frac{\beta}{N}\sum_{i=1}^{N}\langle \nabla f(\bbx^{r}), \tbx_i^{r+1}-\bx_i^r\rangle\nonumber\\
	& =  \frac{\beta}{N}\sum_{i=1}^{N}(h(\tbx^{r+1}_i) - h(\bx_i^r)) +  \frac{\beta}{N}\sum_{i=1}^{N} \langle \by^r_i, \tbx_i^{r+1}-\bx_i^r\rangle \nonumber\\
	& \quad +  \frac{\beta}{N}\sum_{i=1}^{N}\langle \nabla f(\bbx^{r}) -\bby^r, \tbx_i^{r+1}-\bx_i^r\rangle + \frac{\beta}{N}\sum_{i=1}^{N}\langle \bby^r-\by^r_i, \tbx_i^{r+1}-\bx_i^r\rangle
	\end{align}
	where in the first equality we have used \eqref{eq:xxtilde:equal}; in the first inequality we used Jensen's inequality,  the relation \eqref{eq:xxtilde:equal},  as well as the fact that, for indicator functions for a convex set, as long as $\bx_i$'s are all feasible, then we have
	\begin{align*}
	\frac{1}{N}\sum_{i=1}^{N}h(\bx_i) = h(\bbx).
	\end{align*}

	Since  $\tbx^{r+1}_i$ is an optimal solution to the following problem
	\begin{align*}
	\min_{x_i\in X} \langle \by^r_i, \bx_i-\bx^r_i\rangle + \frac{1}{2\alpha(t)}\|\bx_i-\bx^r_i\|^2,
	\end{align*}
	and that $\bx^r_i\in X$ is also feasible for the above problem,  we obtain
	\begin{align*}
	&h(\tbx^{r+1}_i) + \langle \by^r_i, \tbx^{r+1}_i - \bx^r_i\rangle + \frac{1}{2\alpha}\|\tbx^{r+1}_i-\bx^r_i\|^2\le h(\bx^{r}_i). 
	\end{align*}
	Rearranging, we obtain
	\begin{align*}
	&h(\tbx^{r+1}_i)- h(\bx^{r}_i)  + \langle \by^r_i, \tbx^{r+1}_i - \bx_i^r\rangle \le -  \frac{1}{2\alpha}\|\tbx^{r+1}_i-\bx^r_i\|^2.
	\end{align*}
	Applying \eqref{eq:gradient:y} and \eqref{eq:barx:difference}, we obtain
	\begin{align}\label{eq:descent:estimate:term}
	&{\rm Term A} + {\rm Term B} \le -\frac{\beta}{2\alpha N}\|\tbx^{r+1}-\bx^r\|^2 \nonumber\\
	& \quad +  \frac{c\beta}{2}\|\nabla f(\bbx^{r}) -\bby^r\|^2+\frac{\beta}{2Nc} \|\tbx^{r+1}-\bx^r\|^2 + \frac{d\beta}{2N}\|\b1\bby^r-\by^r\|^2 + \frac{\beta}{2Nd}\|\tbx^{r+1}-\bx^r\|^2.
	\end{align}
	So overall, combine the above analysis with \eqref{eq:descent}, \eqref{eq:barx:difference} and \eqref{eq:gradient:y}, we obtain
	\begin{align}\label{eq:descent:estimate}
	f(\bbx^{r+1}) + h(\bbx^{r+1}) & \le f(\bbx^{r}) + h(\bbx^r)- \beta\left(\frac{1}{2\alpha N} -\frac{1}{2cN} - \frac{1}{2Nd}-\frac{L{\beta}}{2N}\right)  \|\bv^{r+1}\|^2 \nonumber\\
	&+ \frac{c \hat{L}^2\beta}{2N} \|\bx^r-\b1 \bbx^r\|+ \frac{d\beta}{2N} \|\by^r -\b1\bby^r\|^2\nonumber\\
	& \le f(\bbx^{r}) + h(\bbx^r)- \beta\left(\frac{1}{8\alpha N}  -\frac{L\beta}{2N}\right)  \|\bv^{r+1}\|^2 \nonumber\\
	& \quad + \frac{\hat{L}^2 \alpha \beta}{N} \|\bx^r-\b1 \bbx^r\|^2+ \frac{{2} \alpha \beta }{N} \|\by^r -\b1\bby^r\|^2,
	\end{align}
	where in the last inequality we let $c= 2\alpha$, and $d= {4} \alpha$. The claim is proved. 
\end{proof}

Next we bound the two ascent terms in \eqref{eq:descent:estimate}. 
\begin{lemma} \label{lemma: bound_terms}
	Suppose that Assumptions 1-2 hold true. Fix an outer iteration $t$, then we have the following descent estimate for all $r\ge 0$:
	\begin{align}\label{eq:poential:partial:2}
	& \|\bx^{r+1}-\b1\bbx^{r+1}\|^2 + \frac{1-(1+\gamma)\eta^2}{32 (1+1/\xi)\hat{L}^2}\|\by^{r+1}-\b1\bby^{r+1}\|^2  \nonumber\\
	& \le   \|\bx^r-\bbx^r\|^2  + \frac{1-(1+\gamma)\eta^2}{32 (1+1/\xi)\hat{L}^2} \|\by^{r}-\b1\bby^{r}\|^2  \nonumber\\
	& - \frac{(1-(1+\gamma)\eta^2)}{2} \|\bx^r-\b1\bbx^r\|^2 +  \beta^2(1+1/\gamma + 1/4)\|\bv^{r+1}\|^2\nonumber\\
	& -  \frac{(1-(1+\gamma)\eta^2)(1-(1+\xi)\eta^2)}{32 (1+1/\xi)\hat{L}^2}\|\by^{r}-\b1\bby^{r}\|^2
	\end{align}
	for any constant $\gamma>0$ and $\xi>0$, and for some $\eta\in (0,1)$, where $ \eta $ is chosen to satisfy that $ (1+\gamma)\eta^2 < 1 $ and $ (1 + \xi)\eta^2 < 1 $.
\end{lemma}

\begin{proof} \label{proof:bound_terms}
	First note that from the definition of the weight matrix $\bW$ in \eqref{eq:W}, there must exist a constant $\eta\in (0,1)$ such that
	\begin{align}\label{eq:contraction}
	\|\bW \bx^r- \b1 \bbx^r\|\le \eta \|\bx^r-\b1\bbx^r\|,
	\end{align}
	and similar bounds hold for $\bv$ and $\by$ as well.
	
	From the update rule of $\bx^{r+1}$, and \eqref{eq:xxtilde:equal}, we obtain
	\begin{align}\label{eq:diffx:temp}
	\|\bx^{r+1}-\b1\bbx^{r+1}\|^2 &= \|\bW (\bx^r + \beta \bv^{r+1}) - \b1(\bbx^r + \beta \bbv^{r+1})\|^2\nonumber\\
	& = \|\bW \bx^r - \b1\bbx^r + \beta (\bW \bv^{r+1}-\b1 \bbv^{r+1})\|^2\nonumber\\
	& \le (1+\gamma)\eta^2 \|\bx^r-\bbx^r\|^2 + (1+1/\gamma)\beta^2\|\bv^{r+1}-\b1 \bbv^{r+1}\|^2.
	\end{align}
	Similarly, from the update rule of $\bv$ and \eqref{eq:xxtilde:equal}, we have
	\begin{align}\label{eq:diffv:temp}
	(1+1/\gamma)\beta^2\|\bv^{r+1}-\b1\bbv^{r+1}\|^2 & =(1+1/\gamma)\beta^2 \|\tbx^{r+1}-\bx^r - \b1 (\tbbx^{r+1}-\bbx^r)\|^2\nonumber\\
	& = (1+1/\gamma)\beta^2 \left\|\left(\bI - \frac{\b1\b1^T}{N}\right)(\tbx^{r+1}-\bx^r) \right\|^2\le (1+1/\gamma)\beta^2 \|\bv^{r+1} \|^2.
	\end{align}
	
	From the update rule of $\by$ and \eqref{eq:ybar}, we obtain
	\begin{align*}
	\|\by^{r+1}-\b1\bby^{r+1}\|^2 & = \|\bW\by^r+\nabla g(\bx^{r+1})-\nabla g(\bx^r) - \b1(\bby^r + \overline{\nabla g(\bx^{r+1})}- \overline{\nabla g(\bx^r)})\|^2\nonumber\\
	& \le (1+\xi)\eta^2\|\by^{r}-\b1\bby^{r}\|^2 + (1+1/\xi)\hat{L}^2\|\bx^{r+1}-\bx^r\|^2\nonumber\\
	& = (1+\xi)\eta^2\|\by^{r}-\b1\bby^{r}\|^2 + (1+1/\xi)\hat{L}^2\|\bW (\bx^r+\beta \bv^{r+1}) - \bx^{r}\|^2\nonumber\\
	& = (1+\xi)\eta^2\|\by^{r}-\b1\bby^{r}\|^2 + (1+1/\xi)\hat{L}^2\|\bW \bx^r-\b1 \bbx^r + \b1 \bbx^r- \bx^r + \beta \bW\bv^{r+1}\|^2\nonumber\\
	& \le \|\by^{r}-\b1\bby^{r}\|^2 -(1-(1+\xi)\eta^2)\|\by^{r}-\b1\bby^{r}\|^2 \nonumber\\
	& \quad + 8 (1+1/\xi)\hat{L}^2  \left(2\|\bx^r-\b1 \bbx^r\|^2 + \beta^2\|\bv^{r+1}\|^2\right)
	\end{align*}
	where the last inequality we used the fact that $\|\bW\|\le 1$ and \eqref{eq:contraction}. 
	
	Multiplying both sides the above  by $(1-(1+\gamma)\eta^2)/(32 (1+1/\xi)\hat{L}^2)$, we obtain
	\begin{align}\label{eq:bound:y}
	& \frac{1-(1+\gamma)\eta^2}{32 (1+1/\xi)\hat{L}^2}\|\by^{r+1}-\b1\bby^{r+1}\|^2 \nonumber\\
	& \le \frac{1-(1+\gamma)\eta^2}{32 (1+1/\xi)\hat{L}^2} \|\by^{r}-\b1\bby^{r}\|^2 -\frac{(1-(1+\gamma)\eta^2)(1-(1+\xi)\eta^2)}{32 (1+1/\xi)\hat{L}^2}\|\by^{r}-\b1\bby^{r}\|^2 \nonumber\\
	& \quad + {(1-(1+\gamma)\eta^2) \left( \frac{1}{2} \|\bx^r-\b1 \bbx^r\|^2 +  \frac{\beta^2}{4} \|\bv^{r+1}\|^2 \right)}.
	\end{align}
	Utilizing \eqref{eq:diffx:temp}, \eqref{eq:diffv:temp} and \eqref{eq:bound:y}, we concludes the proof.  
\end{proof}

\subsection{Proof of Theorem \ref{thm:1}}
\begin{proof} \label{proof:thm1}
	{Here, we denote $ F(\bu) : = f(\bu) + h(\bu) $.} First, combine the estimate in Lemma \ref{lemma:decreasing} -- \ref{lemma: bound_terms}, we arrive at
	\begin{align}\label{eq:descent:estimate:final}
	& F(\bbx^{r+1}) + \|\bx^{r+1}-\b1\bbx^{r+1}\|^2 + \frac{1-(1+\gamma)\eta^2}{32 (1+1/\xi)\hat{L}^2(t)}\|\by^{r+1}-\b1\bby^{r+1}\|^2  \nonumber\\
	& \le F(\bbx^{r}) +  \|\bx^r-\bbx^r\|^2  + \frac{1-(1+\gamma)\eta^2}{32 (1+1/\xi)\hat{L}^2(t)} \|\by^{r}-\b1\bby^{r}\|^2\nonumber\\
	& \quad -\beta\left(\frac{1}{8\alpha(t) N}  -\frac{L(t)\beta}{2N}- \beta(1+1/\gamma+1/4)\right)  \|\bv^{r+1}\|^2 - \left(\frac{(1-(1+\gamma)\eta^2)}{2}  -\frac{\hat{L}^2(t) \alpha \beta}{N} \right)\|\bx^r-\b1 \bbx^r\|^2\nonumber\\
	& \quad -\left(\frac{(1-(1+\gamma)\eta^2)(1-(1+\xi)\eta^2)}{32 (1+1/\xi)\hat{L}^2(t)} - \frac{{2} \alpha(t) \beta }{N} \right)\|\by^r -\b1\bby^r\|^2.
	\end{align}

	Note that we have made the dependency on the outer iteration $t$ explicit at this point.
	So clearly, we can make $\alpha(t) = \mathcal{O}(\min\{1/L(t), \; 1/(\hat{L}(t))^2\})$ (or more precisely, according to \eqref{eq:alpha:choice:2}), and by appropriately choosing various constants according to the statement of the theorem, the above becomes
	\begin{align}\label{eq:descent:P}
	P(\bw^{r+1};t) - P(\bw^r;t) &\le - \frac{c_1}{\alpha(t)}\|\bv^{r+1}\|^2 - c_2\|\bx^r-\b1 \bbx^r\|^2 -\frac{c_3}{\hat{L}^2(t)}\|\by^r -\b1\bby^r\|^2
	\end{align}
	where $c_1, c_2, c_3$ are defined below:
	\begin{subequations}
			\begin{align}
		c_1& = \beta\left(\frac{1}{8 N}  -\frac{L(t)\alpha(t) \beta}{2N}- \alpha(t)\beta(1+1/\gamma+1/4)\right)\label{eq:c1:a}\\
		c_2 & = \left(\frac{(1-(1+\gamma)\eta^2)}{2}  -\frac{\hat{L}^2(t) \alpha(t) \beta}{N} \right)\label{eq:c2:a}\\
		c_3 & = \left(\frac{(1-(1+\gamma)\eta^2)(1-(1+\xi)\eta^2)}{32 (1+1/\xi)} - \frac{{2} \alpha(t) \beta \hat{L}^2(t)}{N} \right).\label{eq:c3:a}
		\end{align}
	\end{subequations}

	Let us take $\beta=1/2$ and $\alpha(t)$ satisfies \eqref{eq:alpha:choice:2}. Then we can verify that 
	\begin{align*}
	c_1>\frac{1}{{32}N}, \quad c_2	>\frac{(1-(1+\gamma))\eta^2}{4}>0, \quad c_3 \ge  {\frac{(1-(1+\gamma)\eta^2)(1-(1+\xi)\eta^2)}{{64} (1+1/\xi)}>0.}
	\end{align*}

	Next let us characterize the optimality condition for the proposed algorithm. 
	Using the descent lemma, it is easy to show that after $R(t)$ inner iterations, the following holds
	\begin{align*}
	&\frac{1}{R(t)}\sum_{r=0}^{R(t)-1}\left(\frac{c_1}{\alpha(t)} \|\bv^{r+1}\|^2 + c_2\|\bx^r-\b1 \bbx^r\|^2 + \frac{c_3}{\hat{L}^2(t)}\|\by^r -\b1\bby^r\|^2\right)
	\le \frac{P(\bw^{0};t) - \underline{f}}{R(t)}.
	\end{align*}
	Dividing both sides by $\alpha(t)$, we have  
	\begin{align}
	&\frac{1}{R(t)}\sum_{r=0}^{R(t)-1}\left(\frac{c_1}{\alpha^2(t)} \|\bv^{r+1}\|^2 + \frac{c_2}{\alpha(t)}\|\bx^r-\b1 \bbx^r\|^2 + \frac{c_3}{\alpha(t)\hat{L}^2(t)}\|\by^r -\b1\bby^r\|^2\right)\nonumber\\
	&\le \frac{P(\bw^{0};t) - P(\bw^{R(t)};t)}{R(t)\alpha(t)} =(P(\bw^{0};t)-P(\bw^{R(t)};t))\epsilon  \le (P(\bw^{0};t)-\underline{f})\epsilon\label{eq:intermediate:R}
	\end{align}
	where we have used the assumption that
	$$\alpha(t) R(t) = 1/\epsilon.$$ 
	
	By using the choice of $\alpha(t)$ in \eqref{eq:alpha:choice:2}, and the fact that 
	$$\frac{N(1-(1+\gamma)\eta^2)}{2\hat{L}^2(t)}\ge \frac{N(1-(1+\gamma)\eta^2)(1-(1+\xi)\eta^2)}{{64} (1+1/\xi)\hat{L}^2(t)},$$
	we have $\alpha(t) \le 1$ and 
	\begin{align*}
	\alpha(t) \hat{L}^2(t)&\le \min\bigg\{\frac{\hat{L}^2(t)}{{8} N}\left(\frac{1}{\hat{L}(t)/2N+(1/\gamma+5/4)}\right), \frac{N(1-(1+\gamma)\eta^2)}{2},\nonumber\\
	& \quad  \frac{N(1-(1+\gamma)\eta^2)(1-(1+\xi)\eta^2)}{{64} (1+1/\xi)},\hat{L}^2(t) \bigg\}\\
	& \le \frac{N(1-(1+\gamma)\eta^2)}{2}.
	\end{align*}
	These results imply that 
	\begin{align*}
	&\frac{1}{R(t)}\sum_{r=0}^{R(t)-1}\left({c_1} \|\bv^{r+1}/\alpha(t)\|^2 + {c_2}\|\bx^r-\b1 \bbx^r\|^2 + \frac{2 c_3}{{N(1-(1+\gamma)\eta^2)}}\|\by^r -\b1\bby^r\|^2\right)\nonumber\\
	&\le \frac{P(\bw^{0};t) - \underline{f}}{R(t)\alpha(t)} = (P(\bw^{0};t)-\underline{f})\epsilon.
	\end{align*}
	This completes the claim. 
\end{proof}

\subsection{Proof of Theorem \ref{thm:2}}
\noindent{\bf Proof.} \label{proof: thm2}
The first step is to stitch different stages of the algorithm together, and estimate the descent of  the potential function across different stages. 
First, from Assumption 1 we note that $\hat{L}(t+1)\ge \hat{L}(t)$, and $L(t+1)\ge L(t)$, because from stage $t$ to stage $t+1$, the radius of the feasible set increases. It follows that the potential function \eqref{eq:potential} also decreases as $t$ increases:
\begin{align*}
P(\bw;t+1)\le P(\bw; t), \forall~\bx\in X(t), \forall~\bw.
\end{align*}
Combining the above with the fact that $\bw$ in the last iteration of $t$th round is the same as the initial condition of $(t+1)$th round, we have
\begin{align*}
P(\bw^0;t+1)\le P(\bw^{R(t)}; t).
\end{align*}
By applying \eqref{eq:intermediate:R}, \eqref{eq:c1:a}, and then averaging $T$ stages of total descent, we can obtain{\small
\begin{align*}
&\frac{1}{T}\sum_{t=1}^{T}\frac{1}{R(t)}\sum_{r=0}^{R(t)-1}\frac{1}{32 N} \|\bv^{r+1}(t)/\alpha(t)\|^2 + \frac{(1-(1+\gamma)\eta^2)}{4}\|\bx^r(t)-\b1 \bbx^r(t)\|^2  + \frac{1-(1+\xi)\eta^2}{32N(1+1/\xi)}\|\by^r(t) -\b1\bby^r(t)\|^2\nonumber\\
&\le \frac{(P(\bw^{0};0)-P(\bw^{R(T)}, T))\epsilon}{T}\le \frac{(P(\bw^{0};0)-\underline{f})\epsilon}{T}.
\end{align*}}
This implies that there exists $t\in[0, T]$, such that $\widehat{\mbox{avg-gap}}(t) = \mathcal{O}(\epsilon)$, where $\widehat{\mbox{avg-gap}}(t)$ is defined in \eqref{eq:gap:constrained}. 

Despite the fact that we have obtained an estimate of $\widehat{\mbox{avg-gap}}(t)$, it is not the final result yet. The reason is that, as we have mentioned at the end of Sec. \ref{sub:optimality}, the stationarity gap measure $\widehat{\mbox{avg-gap}}(t)$ is not the same as ${\mbox{avg-gap}}(t)$, unless $\tbx^{r+1}_i$ computed in \eqref{eq:vupdate1} {\it does not} touch the boundary of $X(t)$ for all $i$ and all $r\in[0, R(t)-1]$.  
Fortunately, the following result shows that, there exists an upper bound $T^*$ such that, for some  $t\in [1, T^*+1]$, the iteration generated in stage $t$ will not touch the boundary of $X(t)$. The precise estimate of such $T^*$ depends on various constants of the algorithm, especially on how fast the radius of the set $X(t)$ increases. Different choices and the associated convergence results will be discussed after we present the result below.

For a given stage, suppose that a sufficiently large $R(t)$ iterations are performed, which satisfies $R(t) \alpha(t)=\epsilon$. Let $r*\in [0, R(t)]$ such that either one of the following three cases happens:
\begin{itemize}
	\item {\bf Case 1)} The following event happens:
	$$\mbox{Dist}(\bbx^{r*}, X(t))\le \frac{v^{t}-v^{t-1}}{2}.$$ 
	That is, the average sequence is ``close" to the boundary of the $t$th feasible set;
	\item {\bf Case 2)} The following event happens:
	\begin{align*}
	\mbox{Dist}(\bx^{r*}_i, X(t))\le \frac{v^{t}-v^{t-1}}{4},\; \mbox{for some}\; i\in[N].
	\end{align*}	 
	That is, there exists some node $i\in[N]$ such that the iterate $\bx^{r*}_i$ is close to the boundary. 
	\item {\bf Case 3)} The following event happens:
	\begin{align*}
	\mbox{Dist}(\tbx^{r*}_i, X(t))\le \frac{v^{t}-v^{t-1}}{8},\; \mbox{for some}\; i\in[N].
	\end{align*}	 
\end{itemize}
If for a given stage $t$, and for all its iterations $r\in[0,R(t)]$, none of Case 1 -- Case 3 happens, then it means that none of  $\bx^{r+1}$, $\bbx^{r+1}$ and $\tbx^{r+1}$ gets close to the boundary of the feasible set $X(t)$. It is then clear that in such a stage $t$ the proposed algorithm runs a sequence of {\it unconstrained} distributed updates. The question to be answered is when this kind of stage will appear. 

Let us then analyze, when Case 1, Case 2 or Case 3 happens, the amount of descent of the potential function. 

First, suppose that either Case 1 or Case 2 happens. Note that we have the following lower bound on the sum of squares of the  direction $\bv^{r+1}$:
\begin{align*}
\sum_{r=0}^{R(t)-1}\frac{\beta^2}{N}\|\bv^{r+1}\|^2& = \sum_{r=0}^{R(t)-1}\frac{\beta^2}{N}\|\tbx^{r+1}-\bx^r\|^2 \ge \sum_{r=0}^{r^*-1}\frac{\beta^2}{N}\|\tbx^{r+1}-\bx^r\|^2 \stackrel{\eqref{eq:barx:difference}}\ge \sum_{r=0}^{r^*-1}\|\bbx^{r+1}-\bbx^r\|^2 \nonumber\\
&\stackrel{(i)}\ge \frac{1}{r^*}\left\|\sum_{r=0}^{r^*-1}(\bbx^{r+1}-\bbx^r)\right\|^2\ge  \frac{1}{R(t)}\|\bbx^{r*}- \bbx^0\|^2.
\end{align*}
where in $(i)$ we have applied the Jensen's inequality.
Since we know that at stage $t$, the initial solution is feasible for the stage $t-1$, then we have $\bbx^0\in X(t-1)$. This means that if Case 1 happens, we have
\begin{align*}
\|\bbx^{r*}- \bbx^0\|\ge \frac{v^{t}-v^{t-1}}{4}.
\end{align*}
This implies that 
\begin{align*}
- \sum_{r=0}^{R(t)-1}\frac{c_1}{\alpha(t)}\|\bv^{r+1}\|^2 = - \sum_{r=0}^{R(t)-1}\frac{c_1}{\alpha(t)}\|\tbx^{r+1}-\bx^r\|^2\le -\frac{c_1 N (v^{t}-v^{t-1})^2}{16 R(t)\alpha(t)\beta^2}.
\end{align*}
Second, if Case 1 never happens (the average stays close to $X(t-1)$), but Case 2 happens, this means that $\|\bx^{r*}_i - \bar{\bx}^{r*}\|\ge \frac{1}{4}(v^t-v^{r-1})$ for some $i\in[N]$. 
Then we have
\begin{align*}
-c_2\sum_{r=0}^{R(t)-1}\|\bx^{r+1} - \b1 \bbx^{r+1}\|^2 \le - c_2\|\bx^{r*}-\b1 \bbx^{r*}\|^2\le - \frac{c_2(v^{t}-v^{t-1})^2}{16}.
\end{align*}
Third, if neigher Case 1 nor Case 2 happens, but Case 3 happens, then we have 
\begin{align*}
\|\bv^{r*}\|^2 = \|\tbx^{r*}-\bx^{r*}\|^2 \ge \frac{(v^{t}-v^{t-1})^2}{64}. 
\end{align*}

Therefore, suppose that at stage $t$, at least one of  Case 1 -- Case 3 happen, and let us denote this event as  {[\textsf{Case 1-Case 3]$(t)$}}.  We can apply \eqref{eq:descent:P} and the above estimates and obtain the following (note $c_2<1$)
\begin{align*}
P(\bw^{R};t) - P(\bw^0;t) &\le \sum_{r=0}^{R(t)-1}- \frac{c_1}{\alpha(t)}\|\bv^{r+1}\|^2 - c_2\|\bx^{r+1}-\b1 \bbx^{r+1}\|^2 \\
& \le - \min\left\{\frac{c_1 N}{R(t)\alpha(t)\beta^2}, c_2, 1/4\right\}\frac{(v^{t}-v^{t-1})^2}{16} \nonumber\\
& = - \min\left\{4 N {c_1\epsilon}, c_2, 1/4\right\}\frac{(v^{t}-v^{t-1})^2}{16}
\end{align*}
where the second equality uses the assumption that the inner iteration limit $R(t)$ is chosen such that $1/\epsilon = R(t)\alpha(t)$, and $\beta=1/2$. Further, suppose that {[\textsf{Case 1-Case 3]$(t)$}} does not happen, then according to the descent estimate \eqref{eq:descent:P}, we have $P(\bw^{R};t) - P(\bw^0;t) \le 0$. That is, we know that the potential function will be reduced, but in this case we do not have any explicit estimate of the descent.

Suppose the event  {[\textsf{Case 1-Case 3]$(t)$} happens in $T$ consecutive stages. Then we have
	\begin{align*}
	P(\bw^{R(T)};T) - P(\bw^0;1) & = \sum_{t=1}^{T} P(\bw^{R(t)};t) - P(\bw^0;t)  \nonumber\\
	& \le - \min\left\{{4 N c_1\epsilon}, c_2,1/4\right\}\frac{\frac{1}{T}\sum_{t=1}^{T}(v^{t}-v^{t-1})^2}{16} T\nonumber\\
	& \le - \epsilon\frac{\frac{1}{T}\sum_{t=1}^{T}(v^{t}-v^{t-1})^2}{128} T
	\end{align*}
	where the second inequality holds because the bound of $c_1$ given in \eqref{eq:c1}, and when $\epsilon/8 \le c_2$ and $\epsilon/8\le 1/4$.
	
	This means that the event  {[\textsf{Case 1-Case 3$](t)$}} can happen at most ${{T}}^*$  {\it consecutive} stages, where ${T}^*$ satisfies 
	\begin{align*}
	{T}^*
	\le \frac{1}{\epsilon}\frac{128 ({P}(\bw^0;1)-\underline{f})}{{\frac{1}{T^*}\sum_{t=1}^{T^*}(v^{t}-v^{t-1})^2}}.
	\end{align*}
	The claim is proved. 
	\hfill $\blacksquare$

	\subsection{Proof of Corollary \ref{eq:cor:poly}}
	\begin{proof} \label{proof:cor1}
		Let us suppose that at each stage we increase $v^{t}$ in a linear rate, that is, $v^t= t \times d$ for some constant $d>0$. Then the average growth rate ${\frac{1}{T^*}\sum_{t=1}^{T^*}(v^{t}-v^{t-1})^2}= d^2$. Therefore $T^*$ in \eqref{eq:hatT} can be bounded by the following
		\begin{align*}
		T^* =\left\lceil \frac{128 ({P}(\bw^0;1)-\underline{f})}{\epsilon{d^2}}\right\rceil.
		\end{align*}
		On the other hand, $R(t) = 1/(\epsilon \alpha(t)) = \mathcal{O}(\max\{\hat{L}^2(t), \hat{L}(t)\}/\epsilon)$. Let us define an increasing growth function $\mbox{gr}(v):=\max\{\hat{L}(R(0,v)), \hat{L}^2(R(0,v))\}$, which characterizes how the Lipschitz constants grows with the radius $v$ of the feasible set. Then the total iteration required is in the following order
		\begin{align*}
		\left\lceil \frac{128  ({P}(\bw^0;1)-\underline{f})}{\epsilon{d^2}}\right\rceil \times \frac{1}{\epsilon} \mbox{gr}\left(d \times \left\lceil \frac{128  ({P}(\bw^0;1)-\underline{f})}{\epsilon{d^2}}\right\rceil \right).
		\end{align*}

		If $f_i(\bx)$ is a $Q$th order polynomial in the form of \eqref{eq:poly}, then it is easy to verify that when restricting this function to the feasible set $B(d\times t, \mathbf{0})$, we must have $\hat{L}(t) = \mathcal{O}((d\times t)^{Q-2})$, and $L(t) = \mathcal{O}((d\times t)^{Q-2})$, where the big O notation hides the dependency on the size of the network $N$ and the coefficients $\{\sigma_{ij}\}$. 
		Summarizing the above, by \eqref{eq:total:complexity}, the total number of computation/communication rounds required is in the following order
		\begin{align*}
		\mathcal{O}\left(\frac{1}{\epsilon}\sum_{t=1}^{\frac{1}{\epsilon d^2}} {\hat{L}^2(t)}\right) = \mathcal{O}\left(\frac{1}{\epsilon}\sum_{t=1}^{\frac{1}{\epsilon d^2}} {(d\times t)^{2Q-4}}\right) = \mathcal{O}\left(\frac{1}{\epsilon} \frac{1}{\epsilon d^2} (\frac{1}{\epsilon d})^{2Q-4}\right).
		\end{align*}
		Therefore, suppose that $d$ is a fixed constant (independent of $\epsilon$), then to achieve $\epsilon$-stationary solution of the non-convex decentralized optimization problem \eqref{eq:problem:main} defined in \eqref{eq:gap:unconstrained}, the total number of local computation as well as the total rounds of communication required is at most $\mathcal{O}({1}/{\epsilon^{2Q-2}})$.
		
		Alternatively, one can also optimize the choice of $d$, and let $d=1/\sqrt{\epsilon}$. In this case, the total stage $T$ becomes a constant that is independent of $\epsilon$, and the total computation/communication rounds required is in the following order
		\begin{align*}
		\mathcal{O}\left(\frac{1}{\epsilon}{\left(\frac{1}{\sqrt{\epsilon}}\right)^{2Q-4}}\right) = \mathcal{O}\left(\frac{1}{\epsilon^{Q-1}}\right).
		\end{align*}
		The claim is proved.
	\end{proof}

\end{document}